\documentclass{amsart}
\usepackage[a4paper, margin=1in]{geometry}
\usepackage{amssymb}
\usepackage[hyphens]{url}
\usepackage{amsmath}
\usepackage{amsthm}
\usepackage{mathrsfs}
\usepackage[title]{appendix}
\usepackage[utf8]{inputenc}
\usepackage[english]{babel}
\usepackage{cite}
\usepackage{hyperref}
\usepackage{cleveref}
\usepackage[section]{placeins}
\RequirePackage{textgreek}
\usepackage{enumitem}
\usepackage{bm}
\usepackage{nicefrac}
\usepackage{mathabx}
\usepackage[all,cmtip]{xy}
\usepackage{caption}
\usepackage{float}
\usepackage{mathtools}
\usepackage{xcolor}
\usepackage{xspace}
\usepackage{stmaryrd}
\usepackage{dsfont}

\hypersetup{
    colorlinks=true,
    linkcolor=blue,
    citecolor=cyan,
    urlcolor=red
}

\hypersetup{linktoc=all}
\numberwithin{equation}{subsection}

\newtheorem{thm}{Theorem}[section]
\newtheorem{lem}[thm]{Lemma}
\newtheorem{prop}[thm]{Proposition}
\newtheorem{cor}[thm]{Corollary}

\newtheorem{defn}[thm]{Definition}

\theoremstyle{nonumberplain}

\newtheorem{prob}[thm]{Problem}

\newtheorem{fct}[thm]{Fact}

\theoremstyle{remark}

\newtheorem{exm}[thm]{Example}

\author{Nachi Avraham-Re'em}
\address{Department of Mathematics, Technion, Israel}
\email{nachi.avraham@gmail.com}

\author{Michael Bj\"{o}rklund}
\address{Department of Mathematics, Chalmers University, G\"{o}teborg, Sweden}
\email{micbjo@chalmers.se}

\thanks{The research was supported by the Knut and Alice Wallenberg Foundation (KAW 2021.0258).}

\title{On stationary actions of locally compact groups and their Radon--Nikodym cocycles}



\begin{document}

\begin{abstract}
We study stationary actions of locally compact measured groups through the structure and regularity of their Radon--Nikodym cocycles.

We start with two dynamical consequences of stationarity. Extending a theorem of Furstenberg--Glasner from discrete groups to noncompact lcsc groups, we show that every stationary action is conservative. Thus stationary actions are never of type \(\mathrm{I}\). We then show that an ergodic stationary action admitting an absolutely continuous invariant \(\sigma\)-finite measure is in fact probability preserving. Thus stationary actions are never of type \(\mathrm{II}_{\infty}\). Using a construction of Katznelson--Weiss and Vaes--Verjans, we show that if a group admits a stationary action of type \(\mathrm{III}_{1}\), then it admits stationary actions of every type \(\mathrm{III}_{\lambda}\), \(0\leq\lambda\leq 1\).

The second part concerns the regularity of the Radon--Nikodym cocycle. We introduce the harmonic majorant on normalized positive harmonic functions, which gives Harnack-type control of Radon--Nikodym derivatives of stationary actions. For compactly supported probability measures with an \(L^{p}\)-density for some \(p>1\), we prove that the harmonic majorant is finite and locally bounded. As a consequence, such measured groups admit a universal compact Radon--Nikodym model: a single compact \(G\)-space with a continuous cocycle into which stationary actions can be realized, so that the ambient cocycle gives a version of its Radon--Nikodym cocycle. This strengthens the Mackey--Varadarajan compact model theorem by incorporating the Radon--Nikodym cocycle into the model.

By contrast, we construct a random walk on the real affine group whose Poisson boundary fails Kaimanovich's \(\mathrm{SAT}^{\ast}\) property: the Poisson kernel is unbounded arbitrarily close to the identity. Therefore, Harnack's inequality already fails for positive harmonic functions. In particular, this Poisson boundary admits no topological model with continuous Poisson kernel.
\end{abstract}

\maketitle

\tableofcontents

\section{Introduction and main results}

Let \(\left(G,\theta\right)\) be a measured group, that is, a locally compact second countable group equipped with an admissible probability measure. A stationary \(\left(G,\theta\right)\)-space is a probability space \(\left(X,\mu\right)\) such that \(X\) is a measurable \(G\)-space and \(\mu\) satisfies
\[\theta\ast\mu=\mu.\]

Stationary actions were introduced by Furstenberg in the study of random walks and Poisson boundaries, and they provide a natural nonsingular counterpart to measure preserving actions in settings where invariant probability measures need not exist \cite{furstenberg1963poisson,nevo2002,furstenberg2010stationary,benoist2016random}. When \(G\) is discrete countable, a substantial part of the theory is by now classical. In the locally compact setting, however, questions of recurrence, absolutely continuous invariant measures, and regularity of the Radon--Nikodym cocycle become genuinely analytic and topological in nature. The purpose of this paper is to study stationary actions of locally compact groups through the structure and regularity of their Radon--Nikodym cocycles.

\smallskip

We begin with conservativity. Basic aspects of conservativity for nonsingular actions of locally compact groups were clarified recently in \cite{avraham2024hopf}, and one form of this property is Poincar\'{e} recurrence, by which for every positive measure set \(A\) in \(X\), there are unboundedly many \(g\in G\) satisfying \(\mu\left(g.A\cap A\right)>0\). Furstenberg--Glasner proved that stationary actions of countably infinite groups are Poincar\'{e} recurrent \cite[\S1]{furstenberg2013recurrence}. Our first main result extends the recurrence of stationary actions to locally compact groups.

\begin{thm}\label{mthm:conservative}
Every stationary action of a noncompact measured group is conservative.
\end{thm}

Viewed through the Krieger classification of nonsingular actions, this theorem says that ergodic stationary actions of noncompact measured groups are never of type \(\mathrm{I}\). The second structural result concerns absolutely continuous invariant measures. For Poisson boundaries it is well-known that they are absent (see \cite[Theorem 3.2.1]{kaimanovich1995poisson}); we show that this is in fact a general feature of stationary actions.

\begin{thm}\label{mthm:typeIII}
An ergodic stationary action is either probability preserving or of type \(\mathrm{III}\), namely it admits no absolutely continuous invariant \(\sigma\)-finite measure.
\end{thm}

Thus ergodic stationary actions never exhibit type \(\mathrm{II}_{\infty}\). This gives a first rigidity statement for stationary dynamics. Later we complement it by adapting a construction of Katznelson--Weiss~\cite[\S5]{katznelson1991classification} (for \(G=\mathbb{Z}\)) and Vaes--Verjans~\cite{vaes2023orbit} (for general \(G\)) to show that, once a group admits a stationary action of type \(\mathrm{III}_{1}\), it also admits stationary actions of all types \(\mathrm{III}_{\lambda}\), \(0\leq\lambda\leq 1\).

\smallskip

The rest of the paper develops the Radon--Nikodym cocycle as a unifying object behind the analytic and topological aspects of stationary actions. A central role is played by the normalized cone of positive \(\theta\)-harmonic (measurable) functions,
\[\mathbb{H}\left(G,\theta\right):=\left\{u:G\to\mathbb{R}_{>0}:u\ast\theta= u,\,u\left(e_{G}\right)=1\right\},\]
where \(e_{G}\) is the identity element of \(G\), and the \emph{harmonic majorant},
\[\mathfrak{m}_{\theta}\left(g\right):=\sup\left\{u\left(g\right):u\in\mathbb{H}\left(G,\theta\right)\right\},\quad g\in G.\]
The harmonic majorant controls ratios of positive harmonic functions and, through the Radon--Nikodym cocycle, controls the distortion of all stationary actions of \(\left(G,\theta\right)\). We say that \(\left(G,\theta\right)\) satisfies Harnack's inequality if \(\mathfrak{m}_{\theta}\left(g\right)<+\infty\) for every \(g\in G\).

\smallskip

Recall that every Poisson boundary has the SAT property (Strongly Approximately Transitive) \cite{jaworski1995strong, furstenberg2010stationary}. A slightly stronger property was introduced by Kaimanovich under the name SAT\(^{\ast}\) property \cite{Kaimanovich2002SAT}. As we shall explain, in discrete countable groups, every stationary action with the SAT property also has the SAT\(^{\ast}\) property, but this is no longer true for general measured group. Still, finiteness of the harmonic majorant forces the Poisson boundary to satisfy the \(\mathrm{SAT}^{\ast}\) property. We prove Harnack's inequality when \(\theta\) is compactly supported with \(L^{p}\)-density for \(p>1\).\footnote{The individual boundedness of Radon--Nikodym derivatives of stationary actions for measured groups with an \(L_{\mathrm{c}}^{\infty}\left(G\right)\)-density seems to be known (see e.g. \cite[\S3]{nevo2002}, \cite[\S1.3]{nevo2000rigidity}), but we could not find its proof in the literature.} On the other hand, we construct a random walk on the real affine group whose Poisson boundary fails the \(\mathrm{SAT}^{\ast}\) property, and whose Radon--Nikodym kernel produces many non-locally bounded positive harmonic functions. We denote by \(L_{\mathrm{c}}^{p}\left(G\right)\) the space of \(L^{p}\)-integrable functions on \(G\) having compact support.

\begin{thm}\label{mthm:Harnack}
If \(\theta\) has an \(L_{\mathrm{c}}^{p}\left(G\right)\)-density for some \(1<p\leq+\infty\), then Harnack's inequality holds:
\[\mathfrak{m}_{\theta}\left(g\right)<+\infty\text{ for all }g\in G.\]
On the other hand, there exists an admissible probability measure \(\theta\) on the real affine group \(G=\mathrm{Aff}\left(\mathbb{R}\right)\) with Poisson boundary \(\left(\mathbb{R},\mu\right)\), such that every identity neighborhood contains an element \(g\) with
\[\Big\Vert\frac{dg\mu}{d\mu}\Big\Vert_{L^{\infty}\left(\mu\right)}=+\infty.\]
In particular, this Poisson boundary is \(\mathrm{SAT}\) but not \(\mathrm{SAT}^{\ast}\), and \(\left(G,\theta\right)\) fails Harnack's inequality.
\end{thm}

The second part of the theorem exhibits a genuinely locally compact phenomenon: even for Poisson boundaries, the Radon--Nikodym kernel may fail to be locally bounded in every neighborhood of the identity. This shows that continuity of the kernel is a genuine issue in the locally compact setting, and motivates the compact model problem to which we now turn.

\smallskip

We now turn to compact topological models for stationary actions. Compact models for measurable \(G\)-spaces go back to Mackey and Varadarajan \cite{mackey1962point,varadarajan1968geometry}. In the context of Poisson boundaries one also seeks models in which the Poisson kernel is continuous, a theme going back to Furstenberg's seminal work and its later developments \cite{furstenberg1963poisson,furstenberg1963noncomm,kaimanovich1995poisson,furman2002random,benoist2016random}. We study this problem for arbitrary stationary actions and show that an integrability condition on the harmonic majorant yields a universal compact Radon--Nikodym model, that simultaneously realizes the Radon--Nikodym factors of all stationary actions.

\begin{thm}\label{mthm:compactRNmodel}
Let \(\left(G,\theta\right)\) be a measured group such that
\[\mathbb{E}_{\theta}\left[\mathfrak{m}_{\theta}\right]<+\infty.\]
Then it admits a universal compact Radon--Nikodym model: there exists a compact \(G\)-space \(\mathbb{G}\left(G,\theta\right)\) equipped with a continuous \(\theta\)-harmonic cocycle
\[\mathfrak{g}:G\times\mathbb{G}\left(G,\theta\right)\to\mathbb{R}_{>0},\]
such that every stationary \(\left(G,\theta\right)\)-space admits a compact model inside \(\mathbb{G}\left(G,\theta\right)\) on which \(\mathfrak{g}\) restricts to a continuous version of its Radon--Nikodym cocycle.
\end{thm}

We shall prove that the integrability condition of Theorem~\ref{mthm:compactRNmodel} automatically holds in the setting of Theorem~\ref{mthm:Harnack}, and we thus obtain the following general result.

\begin{cor}\label{corr:compactRNmodel}
Every measured group having an \(L_{\mathrm{c}}^{p}\left(G\right)\)-density for some \(1<p\leq+\infty\), admits a universal compact Radon--Nikodym model.
\end{cor}

\subsubsection{Organization}

The paper is organized as follows. In Section~\ref{sct:fund} we recall the basic material on measured groups, nonsingular actions, stationary spaces, and Poisson boundaries. Section~\ref{sct:conserv} contains two proofs of Theorem~\ref{mthm:conservative}. In Section~\ref{sct:acim} we prove Theorem~\ref{mthm:typeIII}, and in Section~\ref{sct:type} we discuss the realization of all types \(\mathrm{III}_{\lambda}\). Section~\ref{app:Harn} introduces the harmonic majorant and proves Theorem~\ref{mthm:Harnack}. Finally, Section~\ref{sct:compRN} is devoted to the topological realization of the Radon--Nikodym factor and the proof of Theorem~\ref{mthm:compactRNmodel}.

\section{Fundamentals}\label{sct:fund}

\subsection{Measured Groups}

The main object of study in this work is a {\bf measured group}, that is a pair \(\left(G,\theta\right)\) consisting of a locally compact second countable (lcsc) group \(G\) and an {\bf admissible} probability measure \(\theta\) on \(G\). By this we mean first that \(\theta\) is a Borel probability measure on \(G\) absolutely continuous with respect to the Haar measure class, and second that \(\theta\) is not supported on a closed proper subsemigroup of \(G\). The second condition is equivalent to that
\[G=\overline{\bigcup\nolimits_{n\geq1}\mathrm{supp}\left(\theta^{\ast n}\right)},\]
where the \(n^{\mathrm{th}}\) convolution of \(\theta\) is defined inductively by \(\theta^{\ast 1}=\theta\) and
\[\theta^{\ast\left(n+1\right)}\left(A\right)=\int_{G}\theta^{\ast n}\left(Ah^{-1}\right)d\theta\left(h\right).\]
Fix a left Haar measure \(m_{G}\) on \(G\), and denote the density of \(\theta\) against \(m_{G}\) by
\[\varphi:=d\theta/dm_{G}.\]
For \(g\in G\) we denote
\[g\theta\left(\cdot\right):=g_{\ast}\theta\left(\cdot\right)=\theta\left(g^{-1}\cdot\right)\quad\text{and}\quad g\varphi\left(\cdot\right):=\varphi\left(g^{-1}\cdot\right).\]
Note that \(g\varphi=dg\theta/dm_{G}\). For a measurable function \(u:G\to\mathbb{R}\) denote by \(u\ast\theta:G\to\mathbb{R}\), when is defined, the measurable function
\[u\ast\theta\left(g\right):=\int_{G}u\left(gh\right)d\theta\left(h\right)=\int_{G}u\left(h\right)\cdot g\varphi\left(h\right)dm_{G}\left(h\right).\]
A function \(u\) is {\bf \(\theta\)-harmonic} if \(u\ast\theta=u\). Denote the space of bounded \(\theta\)-harmonic functions by
\[\mathcal{H}^{\infty}\left(G,\theta\right).\]
Note that \(\mathcal{H}^{\infty}\left(G,\theta\right)\) consists of continuous functions, for if \(u\in\mathcal{H}^{\infty}\left(G,\theta\right)\), then for all \(g,h\in G\),
\[\left|u\left(g\right)-u\left(h\right)\right|\leq\int_{G}\left|u\left(k\right)\right|\cdot\left|g\varphi\left(k\right)-h\varphi\left(k\right)\right|dm_{G}\left(k\right)\leq\left\Vert u\right\Vert_{L^{\infty}\left(G\right)}\cdot\left\Vert g\varphi-h\varphi\right\Vert_{L^{1}\left(G\right)},\]
and the last term is continuous in the variables \(g,h\) by the \(L^{1}\)-continuity of translations.

For measurable functions \(u,v:G\to\mathbb{R}\), define their convolution, when it is well defined, by
\[u\ast v\left(g\right):=\int_{G}u\left(h\right)v\left(g^{-1}h\right)dm_{G}\left(h\right).\]
In particular, with \(\varphi=d\theta/dm_G\) one has
\[u\ast\varphi\left(g\right)=u\ast\theta\left(g\right).\]

\subsection{Nonsingular actions}

Let \(G\) be an lcsc group. A {\bf Borel \(G\)-space} is a standard Borel space \(X\) equipped with a measurable action map \(G\times X\to X\). When the action is clear from the context, we will abbreviate it as \(\left(g,x\right)\mapsto g.x\), and relate to \(g\) as an actual Borel automorphism of \(X\).

A pair of measures \(\mu_{1}\) and \(\mu_{2}\) on \(X\) are equivalent, and we denote it by \(\mu_{1}\sim\mu_{2}\), if they are mutually absolutely continuous. A probability measure \(\mu\) on a Borel \(G\)-space \(X\) is said to be {\bf quasi-invariant} if \(g\mu\sim\mu\) for every \(g\in G\), where \(g\mu\) is the pushforward measure of \(\mu\) by \(g\), given by
\[g\mu\left(A\right)=\mu\left(g^{-1}.A\right).\]
When \(\mu\) is a quasi-invariant measure on \(X\), the probability space \(\left(X,\mu\right)\) is called a {\bf nonsingular \(G\)-space}. For every nonsingular \(G\)-space \(\left(X,\mu\right)\), one has the Radon--Nikodym derivatives
\[\frac{dg\mu}{d\mu}\in L^{1}\left(\mu\right),\quad g\in G,\]
which satisfy the cocycle identities
\[\frac{de_{G}\mu}{d\mu}=1\quad\text{and}\quad\frac{d\left(gh\right)^{-1}\mu}{d\mu}=\frac{dg^{-1}\mu}{d\mu}\circ h\cdot\frac{dh^{-1}\mu}{d\mu}\text{ }\mu\text{-a.s. for all }g,h\in G.\]
These identities hold for each pair \(g,h\) on a \(\mu\)-conull set depending on \(g,h\), making it an \emph{almost cocycle}. This object is not unique in a pointwise manner, and we occasionally need to work with \emph{versions} of it which have some extra regularity. Let us then define this notion properly:

\begin{defn}\label{dfn:version}
A {\bf version} of the Radon--Nikodym cocycle of a nonsingular \(G\)-space \(\left(X,\mu\right)\), is any measurable map \(\varrho:G\times X\to\mathbb{R}_{>0}\), \(\varrho:\left(g,x\right)\mapsto\varrho_{g}\left(x\right)\), satisfying \(\varrho_{g}\left(\cdot\right)=\frac{dg^{-1}\mu}{d\mu}\left(\cdot\right)\) \(\mu\)-a.e. for every \(g\in G\).
\end{defn}

A version of the Radon--Nikodym cocycle may not be a cocycle per se but merely an almost cocycle, in that the cocycle property holds only almost everywhere for each pair of group elements individually. We will use the following terminology (following Becker \cite{becker2013cocycles}) to stress this subtle difference:

\begin{defn}\label{dfn:strversion}
A {\bf strict version} of the Radon--Nikodym cocycle of a nonsingular \(G\)-space \(\left(X,\mu\right)\), is any version \(\varpi:G\times X\to\mathbb{R}_{>0}\), \(\varpi:\left(g,x\right)\mapsto\varpi_{g}\left(x\right)\), which further satisfies the cocycle identity pointwise: 
\[\varpi_{gh}\left(x\right)=\varpi_{g}\left(h.x\right)\cdot\varpi_{h}\left(x\right)\text{ for all }g,h\in G\text{ and all }x\in X.\]
\end{defn}

By the celebrated Mackey cocycle theorem (which applies beyond Radon--Nikodym cocycles), as long as \(G\) is an lcsc group, the Radon--Nikodym cocycle of every nonsingular \(G\)-space admits a strict version (see \cite[Lemma~5.26]{varadarajan1968geometry}). The following fact is well-known:

\begin{fct}\label{fct:unifcont}
Let \(\left(X,\mu\right)\) be a nonsingular \(G\)-space.
\begin{enumerate}
    \item For every version \(\varrho:G\times X\to\mathbb{R}_{>0}\) of the Radon--Nikodym cocycle, the following unitary representation of \(G\) on \(L^{2}\left(\mu\right)\) is weakly (equivalently, strongly) continuous:\footnote{Note that \(\pi\), being \(L^{2}\)-valued, is independent of the choice of the version \(\varrho\).}
    \[\pi:G\to\mathrm{U}\left(L^{2}\left(\mu\right)\right),\quad\pi\left(g\right)f\left(x\right):=\sqrt{\varrho_{g^{-1}}\small(x\small)}\cdot f\left(g^{-1}.x\right).\]
    \item For every \(f\in L^{\infty}\left(\mu\right)\), the following function is uniformly continuous:
    \[P^{\mu}f:G\to\mathbb{R},\quad P^{\mu}f\left(g\right):=\int_{X}f\left(g.x\right)d\mu\left(x\right).\]
\end{enumerate}
\end{fct}

\begin{proof}
For (1), note first that \(\pi\) is independent of the chosen version \(\varrho\), as an \(L^{2}\)-valued map, so choosing a strict version it is straightforward that \(\pi\) is a unitary representation. In order to see that \(\pi\) is weakly continuous, note \(\pi\) forms a homomorphism between Polish groups, so by Pettis' automatic continuity property \cite[Thm.~(9.10)]{kechris2012descriptive} it suffices to show that \(\pi\) is weakly measurable. To see that, we check that for arbitrary \(u,v\in L^{2}\left(\mu\right)\) which are defined measurably pointwise, the function
\[F:G\to\mathbb{R},\quad F\left(g\right):=\left\langle \pi\left(g\right)u,v\right\rangle_{L^{2}\left(\mu\right)},\]
is measurable. For \(C>0\) let the truncation \(\chi_{C}\left(t\right)=-C\cdot\mathbf{1}_{\left(-\infty,-C\right)}+t\cdot\mathbf{1}_{\left[-C,C\right]}+C\cdot\mathbf{1}_{\left(C,+\infty\right)}\), and define
\[F_{C}:G\to\mathbb{R},\quad F_{C}\left(g\right):=\int_{X}\chi_{C}\left(\pi\left(g\right)u\left(x\right)\cdot v\left(x\right)\right)d\mu\left(x\right).\]
Since \(\left(g,x\right)\mapsto g^{-1}.x\) and \(\left(g,x\right)\mapsto\varrho_{g^{-1}}\left(x\right)\) are measurable and the integrand is bounded, \(F_{C}\) is measurable \cite[Thm.~(17.25)]{kechris2012descriptive}. Now fix \(g\in G\). By the Cauchy--Schwarz inequality, \(x\mapsto \pi\left(g\right)u\left(x\right)\cdot v\left(x\right)\) belongs to \(L^{1}\left(\mu\right)\), so by dominated convergence \(F\left(g\right)=\lim_{C\to+\infty}F_{C}\left(g\right)\). Hence \(F\) is measurable.

For (2), fix \(f\in L^{\infty}\left(\mu\right)\). Pick a strict version \(\varrho\) of the Radon--Nikodym cocycle, and for \(g,h\in G\) write
\begin{align*}
\left|P^{\mu}f\left(gh\right)-P^{\mu}f\left(g\right)\right|
&=\left|\int_{X}f\left(gh.x\right)d\mu\left(x\right)-\int_{X}f\left(g.x\right)d\mu\left(x\right)\right|=\left|\int_{X}f\left(g.x\right)\left(\varrho_{h^{-1}}\left(x\right)-1\right)d\mu\left(x\right)\right|\\
&\leq\int_{X}\left|f\left(g.x\right)\right|\left|\varrho_{h^{-1}}\left(x\right)-1\right|d\mu\left(x\right)\leq\left\Vert f\right\Vert_{L^{\infty}\left(\mu\right)}\cdot\left\Vert\varrho_{h^{-1}}-\mathbf{1}\right\Vert_{L^{1}\left(\mu\right)}.
\end{align*}
Now for every \(h\in G\), by H\"{o}lder's inequality we have
\[\left\Vert\varrho_{h^{-1}}-\mathbf{1}\right\Vert_{L^{1}\left(\mu\right)}\leq\left\Vert\sqrt{\varrho_{h^{-1}}}+\mathbf{1}\right\Vert_{L^{2}\left(\mu\right)}\cdot\left\Vert\sqrt{\varrho_{h^{-1}}}-\mathbf{1}\right\Vert_{L^{2}\left(\mu\right)}\leq 2\cdot\left\Vert\sqrt{\varrho_{h^{-1}}}-\mathbf{1}\right\Vert_{L^{2}\left(\mu\right)}=2\cdot\left\Vert\pi\left(h\right)\mathbf{1}-\mathbf{1}\right\Vert_{L^{2}\left(\mu\right)},\]
where the second inequality uses \(\int_{X}\varrho_{h^{-1}}d\mu=1\). As \(\pi\) is strongly continuous, the proof is complete.
\end{proof}

\subsection{Stationary actions}

Let \(\left(G,\theta\right)\) be a measured group. A {\bf stationary \(\left(G,\theta\right)\)-space} is a standard probability space \(\left(X,\mu\right)\), where \(X\) is a Borel \(G\)-space and \(\mu\) has the property that it equals the image measure of \(\theta\otimes\mu\) under the action map. More concretely, \(\mu\) satisfies
\[\theta\ast\mu=\mu,\]
where \(\theta\ast\mu\left(A\right)=\int_{G}g\mu\left(A\right)d\theta\left(g\right)\).

\begin{fct}
If \(\left(X,\mu\right)\) is a stationary \(\left(G,\theta\right)\)-space then \(\mu\) is quasi-invariant, and
\[\int_{G}\frac{dg\mu}{d\mu}\left(x\right)d\theta\left(g\right)=1\text{ for }\mu\text{-a.e. }x\in X.\]
\end{fct}

\begin{proof}
We may assume that \(X\) is a compact metric space and that \(G\) acts by homeomorphisms (see Theorem \ref{thm:Vara} below), hence the action of \(G\) on the space of probability measures on \(X\), given by \(g\nu=g_{\ast}\nu\), is continuous in the weak-\(\ast\) topology (see \cite[\S17.E]{kechris2012descriptive}). By the inner-regularity of \(\mu\) and the continuity of the action of \(G\), it suffices to show that for every compact set \(K\) in \(X\) with \(\mu\left(K\right)=0\), it holds that \(g\mu\left(K\right)=0\) for every \(g\in G\). Fix an arbitrary compact set \(K\) with \(\mu\left(K\right)=0\), and define a function \(f_{K}:G\to\left[0,1\right]\) by \(f_{K}\left(g\right)=g\mu\left(K\right)\). Then \(f_{K}\) is upper-semicontinuous by Portmanteau theorem (see \cite[Cor.~(17.21)]{kechris2012descriptive}). Moreover, by the \(\theta\)-stationarity of \(\mu\), for every \(n\geq 1\) we have
\[0=\mu\left(K\right)=\theta^{\ast n}\ast\mu\left(K\right)=\int_{G}f_{K}\left(g\right)d\theta^{\ast n}\left(g\right).\]
Then \(f_{K}=0\) on a \(\theta^{\ast n}\)-conull set, and by upper-semicontinuity \(f_{K}\equiv 0\) on \(\mathrm{Supp}\left(\theta^{\ast n}\right)\). We thus found that \(f_{K}\equiv 0\) on \(\bigcup_{n\geq 1}\mathrm{Supp}\left(\theta^{\ast n}\right)\), which is dense in \(G\) by admissibility, and by upper-semicontinuity it follows that \(f_{K}\equiv 0\) on all of \(G\). This completes the proof of quasi-invariance of \(\mu\). To verify the desired identity, note that by Fubini's theorem and \(\theta\)-stationarity, for every \(f\in L^{\infty}\left(\mu\right)\) it holds that
\begin{align*}
\int_{X}f\left(x\right)\Big(\int_{G}\tfrac{dg\mu}{d\mu}\left(x\right)d\theta\left(g\right)\Big)d\mu\left(x\right)
&=\int_{G}\int_{X}f\left(x\right)\tfrac{dg\mu}{d\mu}\left(x\right)d\mu\left(x\right)d\theta\left(g\right)\\
&=\int_{G}\int_{X}f\left(g.x\right)d\mu\left(x\right)d\theta\left(g\right)=\int_{X}f\left(x\right)d\mu\left(x\right).\qedhere
\end{align*}
\end{proof}

Every stationary \(\left(G,\theta\right)\)-space \(\left(X,\mu\right)\) is associated with the {\bf Poisson transform}, which is the operator
\[P_{\theta}^{\mu}:L^{\infty}\left(\mu\right)\to \mathcal{H}^{\infty}\left(G,\theta\right),\quad P_{\theta}^{\mu}f\left(g\right):=\int_{X}f\left(g.x\right)d\mu\left(x\right).\]
It is a linear, unital, positivity preserving contraction. Its image indeed lies in \(\mathcal{H}^{\infty}\left(G,\theta\right)\), and by Fact~\ref{fct:unifcont}(2) it consists of uniformly continuous functions. Note that \(\mu\) is \(G\)-invariant precisely when \(P_{\theta}^{\mu}f\) is constant for every \(f\in L^{\infty}\left(X,\mu\right)\). In particular, when \(G\) is Liouville, i.e. \(\mathcal{H}^{\infty}\left(G,\theta\right)\) consists of constant functions, every stationary \(\left(G,\theta\right)\)-space is in fact a probability preserving \(G\)-space.

\smallskip

Every measured group \(\left(G,\theta\right)\) is associated a canonical stationary \(\left(G,\theta\right)\)-space called the {\bf Poisson boundary}, denoted by \(\left(\bm{B},\bm{\nu}\right)\), which is uniquely determined by the property that its Poisson transform
\[P_{\theta}:=P_{\theta}^{\bm{\nu}}:L^{\infty}\left(\bm{B},\bm{\nu}\right)\to\mathcal{H}^{\infty}\left(G,\theta\right)\]
forms a surjective isometry of Banach spaces. It is useful to recall that \(P_{\theta}\) is in fact bi-positivity preserving: if \(u=P_{\theta}f\) and \(u\) is nonnegative, then also \(f\) is nonnegative. This follows from the \emph{Probabilistic Fatou theorem} (see \cite[Thm.~24.14]{Woess2000}). The stationary \(\left(G,\theta\right)\)-space \(\left(\bm{B},\bm{\nu}\right)\) is always ergodic (see e.g. \cite[\S5]{zimmer1978amenable}). The construction of the Poisson boundary is due to Furstenberg \cite{furstenberg1963poisson}; see also \cite[\S3.4]{furman2002random} and the references therein. The following fact is a result of the Hahn--Banach theorem (see \cite[Lemma~2.2]{bjorklund2014five}).

\begin{fct}\label{fct:TotalRN}
The linear subspace of \(L^{1}\left(\bm{B},\bm{\nu}\right)\) spanned by \(\big\{dg\bm{\nu}/d\bm{\nu}:g\in G\big\}\) is dense.
\end{fct}

\section{Conservativity and Poincar\'{e} recurrence}\label{sct:conserv}

We recall the basics of conservativity and Poincar\'{e} recurrence in the context of lcsc groups following \cite{avraham2024hopf}. Let \(G\) be an lcsc group with a left Haar measure \(m_{G}\), and let \(\left(X,\mu\right)\) be a nonsingular \(G\)-space. For a Borel set \(A\subset X\) define
\[G_{A}:=\{g\in G:\mu\left(g.A\cap A\right)>0\},\]
and for a point \(x\in A\) define
\[R_{A}\left(x\right):=\left\{g\in G:x\in g.A\right\}.\]
Then \(\left(X,\mu\right)\) is {\bf conservative}, also {\bf Poincar\'{e} recurrent}, if the following equivalent conditions hold:
\begin{enumerate}
    \item For every Borel set \(A\) with \(\mu\left(A\right)>0\), \(G_{A}\) is unbounded (equivalently, \(m_{G}\left(G_{A}\right)=+\infty\)).
    \item For every Borel set \(A\) and \(\mu\)-a.e. \(x\in A\), \(R_{A}\left(x\right)\) is unbounded (equivalently, \(m_{G}\left(R_{A}\left(x\right)\right)=+\infty\)).
    \item There are no positive measure transient sets: For every Borel set \(W\) such that \(R_{W}\left(x\right)\) is bounded (equivalently, \(m_{G}\left(R_{W}\left(x\right)\right)<+\infty\)) for all \(x\in X\), one necessarily has \(\mu\left(W\right)=0\).
    \item There exists (equivalently, for every) positive function \(f\in L^{1}\left(\mu\right)\), such that
    \[\int_{G}f\left(g.x\right)\varpi_{g}\left(x\right)dm_{G}\left(g\right)=+\infty\text{ for }\mu\text{-a.e. }x\in X,\]
    where \(\varpi\) is any strict version of the Radon--Nikodym cocycle of \(\left(X,\mu\right)\).
\end{enumerate}
The first property is called \emph{Poincar\'{e} recurrence}; the second and third properties are called \emph{Halmos' recurrence theorem}; the fourth property is due to Hopf. For the equivalence of the above properties in discrete countable groups see \cite{aaronson1997introduction} and in lcsc groups see \cite{avraham2024hopf}.
\smallskip

We present two proofs of Theorem~\ref{mthm:conservative}, namely that every stationary action of a noncompact measured group is conservative. The first is an adaptation of Furstenberg--Glasner's proof from the discrete case, that establishes Poincar\'{e} recurrence using the martingale convergence theorem. The second proof uses the ergodicity of the Poisson boundary, and it is based on the observation that transient sets give rise to invariant functions on the Poisson boundary.

\subsection{Furstenberg--Glasner's proof}

For a measured group \(\left(G,\theta\right)\), consider the probability space 
\[\left(\Omega,\mathbb{P}\right):=\left(G^{\mathbb{N}},\theta^{\otimes\mathbb{N}}\right),\]
and define the random walk \(\left(z_{n}\right)_{n=0}^{\infty}\) as a sequence of random variables defined on \(\left(\Omega,\mathbb{P}\right)\) by
\[z_{n}\left(\omega\right):=\omega_{0}\omega_{1}\dotsm\omega_{n-1}\text{ for }\omega=\left(\omega_{0},\omega_{1},\dotsc\right)\in\Omega.\]
Denote the (non-invertible) shift transformation
\[\sigma:\Omega\to\Omega,\quad\sigma:\left(\omega_{0},\omega_{1},\omega_{2}\right)\mapsto\left(\omega_{1},\omega_{2},\dotsc\right).\]
Let \(\left(\mathcal{F}_{n}\right)_{n\geq 1}\) be the filtration given by \(\mathcal{F}_{n}=\sigma\left(\omega_{0},\omega_{1},\dotsc,\omega_{n-1}\right)\).

\begin{lem}\label{lem:escape}
If \(\left(G,\theta\right)\) is a noncompact measured group, then for every compact set \(C\) in \(G\),
\[\mathbb{P}\left(\liminf_{n\to\infty}\left\{z_{n}\in C\right\} \right)=0.\]
In particular, for every countable compact exhaustion \(C_{1}\subseteq C_{2}\subseteq\dotsm\) of \(G\) one has
\[\mathbb{P}\left(\Omega_{o}\right)=1,\text{ where }\Omega_{o}:=\bigcap\nolimits_{k\geq 1}\{z_{n}\notin C_{k}\text{ infinitely often}\}.\]
\end{lem}

\begin{proof}[Proof of Lemma~\ref{lem:escape}]
Fix a compact set \(C\subset G\). Let \(\Sigma\left(\theta\right)\) be the semigroup generated by \(\mathrm{supp}\left(\theta\right)\). Since \(\left(G,\theta\right)\) is noncompact, \(\Sigma\left(\theta\right)\) is not relatively compact and in particular it is not contained in the compact set \(C^{-1}C\), so pick some element
\[h:=s_{1}\dotsm s_{m}\in \Sigma\left(\theta\right)\backslash C^{-1}C,\quad s_{1},\dotsc,s_{m}\in\mathrm{supp}\left(\theta\right).\]
Pick an open set \(h\in W\subset G\backslash C^{-1}C\), and by continuity of multiplication there are open sets \(s_{1}\in U_{1},\dotsc,s_{m}\in U_{m}\) so that \(U_{1}\dotsm U_{m}\subseteq W\). In particular we have
\[p:=\theta\left(U_{1}\right)\dotsm\theta\left(U_{m}\right)>0.\]
For \(n\geq 1\) consider the event
\[E_{n}:=\{\omega_{n}\in U_{1},\omega_{n+1}\in U_{2},\dotsc,\omega_{n+m-1}\in U_{m}\}\in\mathcal{F}_{n+m},\]
and note that \(\mathbb{P}\left(E_{n}\right)=p\) while \(E_{n}\) is independent of \(\mathcal{F}_{n}\). We now claim that
\begin{equation}\label{eq:znm}
\mathbb{P}\left(z_{n}\in C,\,z_{n+m}\in C\mid\mathcal{F}_{n}\right)\leq1-\mathbb{P}\left(E_{n}\right)=1-p\text{ for every }n\geq 1.
\end{equation}
Indeed, if \(z_{n}\in C\) and \(E_{n}\) occurs, then \(v:=\omega_{n}\omega_{n+1}\dotsm\omega_{n+m-1}\in U_{1}\dotsm U_{m}\subset G\backslash C^{-1}C\), so that \(v\notin C^{-1}C\) and hence \(z_{n}v\notin C\). This means that \(z_{n+m}=z_{n}v\notin C\), namely \(\{z_{n}\in C,\,z_{n+m}\in C\}\subseteq \{z_{n}\in C\}\cap E_{n}^{\complement}\), concluding \eqref{eq:znm}. Consider the event
\[D_{n}=\left\{z_{n}\in C,z_{n+1}\in C,z_{n+2}\in C,\dotsc\right\}.\]
Then for every \(k\geq 1\) one has
\[\{z_{n}\in C,\,z_{n+m}\in C,\dotsc,z_{n+km}\in C\}\subseteq E_{n}^{\complement}\cap E_{n+m}^{\complement}\dotsm\cap E_{n+\left(k-1\right)m}^{\complement}.\]
Since the events \(E_{n},E_{n+m},\dotsc,E_{n+\left(k-1\right)m}\) depend on disjoint coordinates of \(\omega\), they are independent, and each is independent of \(\mathcal{F}_{n}\). Therefore, \(\mathbb{P}\left(D_{n}\right)\leq\left(1-p\right)^{k}\) for every \(k\geq 1\), and thus \(\mathbb{P}\left(D_{n}\right)=0\). Since \(\liminf_{n\to\infty}\left\{z_{n}\in C\right\} =\bigcup_{n\geq1}D_{n}\), this completes the proof.
\end{proof}

\begin{proof}[Proof of Theorem~\ref{mthm:conservative}]
Let \(\left(X,\mu\right)\) be a stationary \(\left(G,\theta\right)\)-space. Fix a Borel set \(A\subset X\) with \(\mu\left(A\right)>0\). Define the random variables
\[M_{n}:=z_{n}\mu\left(A\right)=\mu\left(z_{n}^{-1}.A\right),\quad n\geq 1.\]
Using the \(\theta\)-stationarity of \(\mu\), for every \(n\geq 1\) we have
\[\mathbb{E}\left[M_{n+1}\mid\mathcal{F}_{n}\right]=\int_{G}\mu\big(\left(z_{n}h\right)^{-1}.A\big)d\theta\left(h\right)=\int_{G}h\mu\left(z_{n}^{-1}.A\right)d\theta\left(h\right)=\mu\left(z_{n}^{-1}.A\right)=M_{n},\]
and therefore \(\left(M_{n}\right)_{n\geq 1}\) is an \(\left(\mathcal{F}_{n}\right)_{n\geq 1}\)-martingale. The \(\theta\)-stationarity of \(\mu\) gives also \(\mathbb{E}\left[M_{1}\right]=\mu\left(A\right)\), so by the martingale convergence theorem there is \(M\in L^{1}\left(\mathbb{P}\right)\) so that \(M_{n}\to M\) \(\mathbb{P}\)-a.s. as \(n\to+\infty\), and \(\mathbb{E}\left[M\right]=\mu\left(A\right)\). Put \(\delta:=\mu\left(A\right)/2\) and consider the event \(E:=\{M>\delta\}\). Since \(\mathbb{P}\left(E\right)>0\), for every \(\omega\in E\) there is \(n\left(\omega\right)\geq 1\) such that \(M_{n}\left(\omega\right)\geq\delta\) for all \(n\geq n\left(\omega\right)\).

We now deduce the conservativity by showing that \(G_{A}=\{g\in G:\mu\left(g.A\cap A\right)>0\) is not relatively compact. Fix an arbitrary compact set \(K\subset G\). Fix a compact exhaustion \(C_{1}\subseteq C_{2}\subseteq \dotsm\) of \(G\) so that every compact set in \(G\) is contained in \(C_{k}\) eventually,\footnote{This is possible by taking balls with respect to a compatible proper metric on \(G\), which exists by Struble's theorem \cite{struble1974metrics}.} and let \(\Omega_{o}\) be as in Lemma~\ref{lem:escape}. Since \(\mathbb{P}\left(\Omega_{o}\right)=1\) we have \(\mathbb{P}\left(E\cap\Omega_{o}\right)=\mathbb{P}\left(E\right)>0\). Choose \(N\geq 1\) so that \(N\delta>1\), and we will construct random times
\[n\left(\omega\right)=n_{1}\left(\omega\right)<n_{2}\left(\omega\right)<\dotsm<n_{N}\left(\omega\right)\text{ for }\omega\in E\cap\Omega_{o},\]
subject to the property that for all \(1\leq m\leq N\),
\[z_{n_{m+1}\left(\omega\right)}\left(\omega\right)\notin Kz_{n_{j}\left(\omega\right)}\left(\omega\right),\quad1\leq j\leq m.\]
Assume \(n_{1}\left(\omega\right)<\dotsm<n_{m}\left(\omega\right)\) are defined for some \(m<N\), and consider the compact set
\[K_{m}:=\bigcup\nolimits_{j=1}^{m}Kz_{n_{j}}\left(\omega\right).\]
Let \(k\geq 1\) be such that \(K_{m}\subseteq C_{k}\). Since \(\omega\in\Omega_{o}\), there can be found \(n>n_{m}\left(\omega\right)\) such that \(z_{n}\left(\omega\right)\notin C_{k}\), hence \(z_{n}\left(\omega\right)\notin K_{m}\). Then set \(n_{m+1}\left(\omega\right):=n\). This completes the construction of \(n_{1}\left(\omega\right)<\dotsm<n_{N}\left(\omega\right)\).

Fix any \(\omega\in E\cap\Omega_{o}\), and for \(1\leq m\leq N\) define \(A_{m}:=z_{n_{m}\left(\omega\right)}\left(\omega\right)^{-1}.A\). Since \(n_{m}\left(\omega\right)\geq n\left(\omega\right)\) we have
\[\mu\left(A_{m}\right)=M_{n_{m}\left(\omega\right)}\left(\omega\right)\geq\delta.\]
Since \(N\delta>1\), there can be found \(1\leq l<m\leq N\) with \(\mu\left(A_{l}\cap A_{m}\right)>0\), and we then claim that
\[g:=z_{n_{m}\left(\omega\right)}\left(\omega\right)z_{n_{l}\left(\omega\right)}\left(\omega\right)^{-1}\in G_{A}\cap G\backslash K.\]
First, by construction we have \(g\notin K\). To see that \(g\in G_{A}\), write
\[0<\mu\left(A_{l}\cap A_{m}\right)=\mu\big(z_{n_{l}\left(\omega\right)}\left(\omega\right)^{-1}.\left(A\cap g^{-1}.A\right)\big),\]
and by the nonsingularity also \(\mu\left(A\cap g^{-1}.A\right)>0\), hence \(\mu\left(g.A\cap A\right)>0\), and thus \(g\in G_{A}\).
\end{proof}

\subsection{A Poisson boundary proof}

Let us start with a general lemma on Poisson boundaries.

\begin{lem}\label{lem:infs}
Let \(\left(G,\theta\right)\) be a noncompact measured group and let \(\left(\bm{B},\bm{\nu}\right)\) be its Poisson boundary. Suppose \(\phi\in L^{\infty}\left(\bm{\nu}\right)\) is a nonnegative function with the property that the following function is constant:
\[\Phi:\bm{B}\to\mathbb{R}_{\geq0}\cup\{+\infty\},\quad\Phi\left(b\right):=\int_{G}g\phi\left(b\right)dm_{G}\left(g\right).\]
Then either \(\phi=0\) \(\bm{\nu}\)-a.e. and hence \(\Phi\equiv 0\), or otherwise \(\Phi\equiv+\infty\).
\end{lem}

\begin{proof}[Proof of Lemma~\ref{lem:infs}]
Set the constant \(\alpha\equiv\Phi\). If \(\alpha=0\), then since \(\phi\geq0\) and \(\Phi\left(b\right)=0\) for \(\bm{\nu}\)-a.e. \(b\in\bm{B}\), we have \(\int_{G}g\phi\left(b\right)dm_{G}\left(g\right)=0\) for \(\bm{\nu}\)-a.e. \(b\). By Fubini's theorem this implies that \(g\phi=0\) \(\bm{\nu}\)-a.e. for \(m_{G}\)-a.e. \(g\in G\), namely \(g\bm{\nu}\left(\phi>0\right)=0\) for \(m_{G}\)-a.e. \(g\). By nonsingularity it follows that \(\bm{\nu}\left(\phi>0\right)=0\), so \(\phi=0\) \(\bm{\nu}\)-a.e. Assume then \(\alpha>0\) and we will show that \(\alpha=+\infty\).

Fix an arbitrary nonnegative function \(\rho\in C_{\mathrm{c}}\left(G\right)\) with \(\left\Vert\rho\right\Vert_{L^{1}\left(G\right)}=1\), and let the functions
\[\phi_{\rho}\left(b\right)=\int_{G}h\phi\left(b\right)\cdot\rho\left(h\right)dm_{G}\left(h\right)\quad\text{and}\quad\Phi_{\rho}\left(b\right)=\int_{G}g\phi_{\rho}\left(b\right)dm_{G}\left(g\right).\]
Then also \(\Phi_{\rho}\equiv\alpha\), since by Fubini's theorem and left invariance of \(m_{G}\),
\begin{align*}
\Phi_{\rho}\left(b\right)
&=\int_{G}\int_{G}\phi\big(\left(gh\right)^{-1}.b\big)\rho\left(h\right)dm_{G}\left(h\right)dm_{G}\left(g\right)\\
&=\int_{G}\int_{G}\phi\left(g^{-1}.b\right)\rho\left(h\right)dm_{G}\left(g\right)dm_{G}\left(h\right)=\int_{G}\Phi\left(b\right)\rho\left(h\right)dm_{G}\left(h\right)=\alpha.
\end{align*}
For every \(g\in G\) and \(b\in\bm{B}\) we may bound
\begin{align*}
\big|g\phi_{\rho}\left(b\right)-\phi_{\rho}\left(b\right)\big|
&=\Big|\int_{G}\phi\left(h^{-1}g^{-1}.b\right)\rho\left(h\right)dm_{G}\left(h\right)-\int_{G}\phi\left(h^{-1}.b\right)\rho\left(h\right)dm_{G}\left(h\right)\Big|\\
&=\Big|\int_{G}\phi\left(h^{-1}.b\right)\rho\left(gh\right)dm_{G}\left(h\right)-\int_{G}\phi\left(h^{-1}.b\right)\rho\left(h\right)dm_{G}\left(h\right)\Big|\\
&\leq\left\Vert\phi\right\Vert_{L^{\infty}\left(\bm{\nu}\right)}\cdot\int_{G}\left|\rho\left(gh\right)-\rho\left(h\right)\right|dm_{G}\left(h\right).
\end{align*}
The last term is independent of \(b\), and is continuous in its \(g\)-variable. Therefore, for every \(\epsilon>0\) there exists a compact identity neighborhood \(V_{\epsilon}\subset G\) such that \(\left|g\phi_{\rho}\left(b\right)-\phi_{\rho}\left(b\right)\right|<\epsilon\) for all \(g\in V_{\epsilon}\) and \(b\in\bm{B}\).

As \(\phi_{\rho}\) is nonnegative and nonzero (we assumed \(\alpha>0\)), find \(\epsilon>0\) so that \(B_{\epsilon}:=\left\{\phi_{\rho}>\epsilon\right\}\subset\bm{B}\) satisfies \(\bm{\nu}\left(B_{\epsilon}\right)>0\). For every \(g\in V_{\epsilon/2}\), if \(b\in B_{\epsilon}\) then
\[g\phi_{\rho}\left(b\right)\geq\phi_{\rho}\left(b\right)-\epsilon/2>\epsilon/2,\]
hence \(B_{\epsilon}\subset g.B_{\epsilon/2}\). Since by the defining property of the Poisson boundary \(P_{\theta}\) is an isometry,
\[\left\Vert P_{\theta}\mathbf{1}_{B_{\epsilon}}\right\Vert_{L^{\infty}\left(G\right)}=\left\Vert\mathbf{1}_{B_{\epsilon}}\right\Vert_{L^{\infty}\left(\bm{\nu}\right)}=1.\]
Put \(u:=P_{\theta}\mathbf{1}_{B_{\epsilon}}\in\mathcal{H}^{\infty}\left(G,\theta\right)\). We claim that for every \(k\in\mathbb{N}\) the set
\[A_{k}:=\left\{g\in G:u\left(g\right)>1-1/k\right\}\]
is not relatively compact. Indeed, if \(A_{k}\) were relatively compact for some \(k\), then since \(\left\Vert u\right\Vert_{L^{\infty}\left(G\right)}=1\), there is a sequence \(\left(g_{n}\right)\subset A_{k}\) with \(u\left(g_{n}\right)\to1\). Passing to a convergent subsequence \(g_{n}\to g_{o}\in\overline{A_{k}}\), by continuity \(u\left(g_{o}\right)=1\). Since \(0\leq u\leq1\) and \(u\) is \(\theta\)-harmonic, for every \(n\geq 1\) we get
\[1=u\left(g_{o}\right)=\int_{G}u\left(g_{o}h\right)d\theta^{\ast n}\left(h\right),\]
hence \(u\left(g_{o}h\right)=1\) for \(\theta^{\ast n}\)-a.e. \(h\in G\). By continuity of \(u\) and admissibility of \(\theta\), it follows that \(u\equiv 1\), but this contradicts the assumption that \(A_{k}\) is not relatively compact.

Fix an arbitrary \(k\in\mathbb{N}\). Since \(A_{k}\) is not relatively compact, there is \(S\subset A_{k}\) with cardinality \(\#S=k\) such that \(sV_{\epsilon/2}\), \(s\in S\), are pairwise disjoint. Since \(u\left(s\right)=\bm{\nu}\left(s^{-1}.B_{\epsilon}\right)>1-1/k\) for every \(s\in S\), we have
\[\bm{\nu}\big(\bigcap\nolimits_{s\in S}s^{-1}.B_{\epsilon}\big)>0.\]
Let now \(b\in\bigcap_{s\in S}s^{-1}.B_{\epsilon}\). For every \(s\in S\) and every \(g\in sV_{\epsilon/2}\), writing \(g=sv\) with \(v\in V_{\epsilon/2}\), we have \(s^{-1}.b\in B_{\epsilon}\subset v.B_{\epsilon/2}\), hence \(g^{-1}.b=v^{-1}s^{-1}.b\in B_{\epsilon/2}\). Therefore \(g\phi_{\rho}\left(b\right)=\phi_{\rho}\left(g^{-1}.b\right)>\epsilon/2\), and thus
\[\Phi_{\rho}\left(b\right)=\int_{G}g\phi_{\rho}\left(b\right)dm_{G}\left(g\right)\geq\sum\nolimits_{s\in S}\int_{sV_{\epsilon/2}}g\phi_{\rho}\left(b\right)dm_{G}\left(g\right)\geq\#S\cdot\tfrac{\epsilon}{2}\cdot m_{G}\left(V_{\epsilon/2}\right)=k\cdot\tfrac{\epsilon}{2}\cdot m_{G}\left(V_{\epsilon/2}\right).\]
Since \(k\in\mathbb{N}\) is arbitrary while \(\epsilon\) is fixed, it follows that \(\Phi_{\rho}\) is unbounded, hence \(\alpha\equiv\Phi_{\rho}\equiv+\infty\).
\end{proof}

\begin{proof}[Proof of Theorem \ref{mthm:conservative}]
Let \(\left(X,\mu\right)\) be a stationary \(\left(G,\theta\right)\)-space, and assume \(W\) is a Borel set such that
\(m_{G}\left(R_{W}\left(x\right)\right)<+\infty\) for every \(x\in X\). For \(r>0\) set
\[X_{r}:=\left\{x\in X:m_{G}\left(R_{W}\left(x\right)\right)\leq r\right\}\text{ and }W_{r}:=W\cap X_{r}.\]
Note that \(X_{r}\) is \(G\)-invariant (by the left invariance of \(m_{G}\) with the general property \(R_{W}\left(g.x\right)=gR_{W}\left(x\right)\)). Since \(W\) is transient, \(X_{r}\nearrow X\) modulo \(\mu\) as \(r\nearrow+\infty\), hence \(W_{r}\nearrow W\) modulo \(\mu\) as \(r\nearrow+\infty\). Then to show that \(\mu\left(W\right)=0\) we will show that \(\mu\left(W_{r}\right)=0\) for every \(r>0\). Fix \(r>0\), and we first claim that \(m_{G}\left(R_{W_{r}}\left(x\right)\right)\leq r\) for every \(x\in X\); indeed, if \(x\in X_{r}\) then trivially \(R_{W_{r}}\left(x\right)\subseteq R_{W}\left(x\right)\), hence \(m_{G}\left(R_{W_{r}}\left(x\right)\right)\leq m_{G}\left(R_{W}\left(x\right)\right)\leq r\), and if \(x\notin X_{r}\) then \(R_{W_{r}}\left(x\right)=\emptyset\) since \(X_{r}\) is \(G\)-invariant, hence  \(m_{G}\left(R_{W_{r}}\left(x\right)\right)=0<r\). Consider the function
\[u:=P_{\theta}^{\mu}1_{W_{r}}\in\mathcal{H}^{\infty}\left(G,\theta\right).\]
Using the defining property of the Poisson boundary, let \(\phi\in L^{\infty}\left(\bm{\nu}\right)\) be such that \(u=P_{\theta}\phi\), and define
\[\Phi\left(b\right):=\int_{G}g\phi\left(b\right)dm_{G}\left(g\right).\]
Since \(P_{\theta}\) is bi-positivity preserving, \(\phi\) and \(\Phi\) are nonnegative, and \(\Phi\) is \(G\)-invariant by the left invariance of \(m_{G}\). By Fubini's theorem we have
\[\int_{\bm{B}}\Phi\left(b\right)d\bm{\nu}\left(b\right)=\int_{G}P_{\theta}\phi\left(g^{-1}\right)dm_{G}\left(g\right)=\iint\nolimits_{G\times X}g1_{W_{r}}\left(x\right)dm_{G}\otimes\mu\left(g,x\right)=\int_{X}m_{G}\left(R_{W_{r}}\left(x\right)\right)d\mu\left(x\right)\leq r,\]
hence \(\Phi\in L^{1}\left(\bm{\nu}\right)\). Since the Poisson boundary \(\left(\bm{B},\bm{\nu}\right)\) is ergodic, \(\Phi\) is a finite constant \(\bm{\nu}\)-a.e. From Lemma~\ref{lem:infs} it follows that \(\Phi\equiv0\), hence \(\phi=0\) \(\bm{\nu}\)-a.e. Therefore \(u=P_{\theta}\phi\equiv0\) (everywhere, as \(u\) is continuous), and evaluating at the identity gives \(0=u\left(e_{G}\right)=\mu\left(W_{r}\right)\).
\end{proof}

\section{Absolutely continuous invariant measure}\label{sct:acim}

Let \(G\) be an lcsc group and \(\left(X,\mu\right)\) a nonsingular \(G\)-space. A Borel \(\sigma\)-finite measure \(m\) on \(X\) is called {\sc a.c.i.m.} of \(\left(X,\mu\right)\) if it is absolutely continuous with respect to \(\mu\) and \(G\)-invariant. The question whether a nonsingular \(G\)-space admits an {\sc a.c.i.m.} is fundamental in ergodic theory and was addressed by many authors. See e.g. \cite{danilenko2023ergodic} and the many references therein. In the common terminology, an ergodic nonsingular \(G\)-space is called {\bf type \(\mathrm{II}_{1}\)} if it admits a finite {\sc a.c.i.m.}; {\bf type \(\mathrm{II}_{\infty}\)} if it admits an infinite {\sc a.c.i.m.}; and {\bf type \(\mathrm{III}\)} if it admits no {\sc a.c.i.m.}

\smallskip

For stationary \(\left(G,\theta\right)\)-spaces more is known. First, it was shown by Bader--Shalom~\cite[Prop.~2.6]{bader2006factor} that every two ergodic \(\theta\)-stationary probability measures on a Borel \(G\)-space \(X\) are either mutually singular or equal. In particular, an ergodic stationary \(\left(G,\theta\right)\)-space of type \(\mathrm{II}_{1}\) must be itself probability preserving. This still leaves the question whether a stationary \(\left(G,\theta\right)\)-space can be of type \(\mathrm{II}_{\infty}\) unanswered. In the following we give a general answer to this, generalizing Theorem~\ref{mthm:typeIII}:

\begin{thm}\label{thm:typeIII}
Let \(\left(G,\theta\right)\) be a measured group. Then every ergodic stationary \(\left(G,\theta\right)\)-space is either itself probability preserving, or otherwise type \(\mathrm{III}\). More generally, every stationary \(\left(G,\theta\right)\)-space \(\left(X,\mu\right)\) admits a unique \(G\)-invariant decomposition \(X=X_{1}\sqcup X_{\theta}\), such that \(\mu\left(\cdot\cap X_{1}\right)\) is \(G\)-invariant and \(\mu\left(\cdot\cap X_{\theta}\right)\) admits no {\sc a.c.i.m.}
\end{thm}

\subsection{Ergodicity of nonsingular actions}

We recall some basics about ergodicity in the general context of nonsingular actions. Let \(\left(X,\mu\right)\) be a nonsingular \(G\)-space. A measurable function \(f:X\to\mathbb{R}\) is considered to be \(G\)-invariant, if \(gf=f\) \(\mu\)-a.e. for every \(g\in G\). Then \(\left(X,\mu\right)\) is {\bf ergodic} if every \(G\)-invariant function is necessarily constant \(\mu\)-a.e. To verify ergodicity, it suffices to check that every \(G\)-invariant \(L^{\infty}\left(\mu\right)\)-function is constant \(\mu\)-a.e. or also that the indicator function of every \(G\)-invariant measurable set is constant \(\mu\)-a.e.\footnote{This is because every \(G\)-invariant function \(f:X\to\mathbb{R}\) produces the \(G\)-invariant sets \(\left\{f\geq c\right\}\) for \(c>0\).}

\smallskip

We will need the ergodic decomposition theorem by Greschonig--Schmidt \cite{greschonig2000ergodic}. For a nonsingular \(G\)-space \(\left(X,\mu\right)\) with a strict version \(\varpi\) of the Radon--Nikodym cocycle, let \(\mathcal{E}^{G}_{\varpi}\left(X\right)\) be the space of ergodic quasi-invariant probability measures on \(X\), such that \(\varpi\) is a version of their Radon--Nikodym cocycle. By \cite[Thm.~5.2]{greschonig2000ergodic} there is a \(G\)-invariant map
\[\upsilon:X\to\mathcal{E}^{G}_{\varpi}\left(X\right),\quad x\mapsto\upsilon_{x},\]
such that
\[\mu\left(\cdot\right)=\int_{X}\upsilon_{x}\left(\cdot\right)d\mu\left(x\right),\]
and, more generally, for every nonnegative Borel function \(f:X\to\mathbb{R}_{\geq 0}\),
\[\upsilon_{x}\left(f\right)=\mathbb{E}_{\mu}\left[f\mid\mathcal{I}\left(G\right)\right]\left(x\right)\quad\text{for \(\mu\)-a.e. }x\in X,\]
where \(\mathcal{I}\left(G\right)\) denotes the \(\sigma\)-algebra of \(G\)-invariant Borel subsets of \(X\). The following lemma shows that the ergodic decomposition for nonsingular actions commutes with passing to an {\sc a.c.i.m.}

\begin{lem}\label{lem:ergdecompacim}
Let \(\left(X,\mu\right)\) be a nonsingular \(G\)-space with ergodic decomposition \(\upsilon:X\to\mathcal{E}_{\varpi}^{G}\left(X\right)\). Then every {\sc a.c.i.m.} \(m\) of \(\left(X,\mu\right)\) admits an ergodic decomposition
\[m\left(\cdot\right)=\int_{X}m_{x}\left(\cdot\right)d\mu\left(x\right),\]
such that \(m_{x}\) is an ergodic {\sc a.c.i.m.} for \(\left(X,\upsilon_{x}\right)\) for \(m\)-a.e. \(x\in X\).
\end{lem}

\begin{proof}[Proof of Lemma~\ref{lem:ergdecompacim}]
It is a basic fact that \(m\) is supported on a \(G\)-invariant set on which it is mutually absolutely continuous with \(\mu\) (see e.g. \cite[Lem.~17.7]{avraham2024hopf}). Replacing \(X\) by this \(G\)-invariant set, we may assume that \(m\sim\mu\). Fix a positive Borel function \(\kappa:X\to\mathbb{R}_{>0}\) with \(\kappa:=dm/d\mu\) \(\mu\)- hence \(m\)-a.e. For \(\nu\in\mathcal{E}_{\varpi}^{G}\left(X\right)\), let \(m_{\nu}\) be the \(\sigma\)-finite measure on \(X\) given by
\[dm_{\nu}=\kappa d\nu.\]
Define the probability measure
\[\widehat{\mu}:=\upsilon_{\ast}\mu\text{ on }\mathcal{E}_{\varpi}^{G}\left(X\right).\]
Then for every nonnegative measurable function
\(\Psi:\mathcal{E}_{\varpi}^{G}\left(X\right)\times X\to\mathbb{R}_{\geq 0}\),
\begin{equation}\label{eq:hatiden}
\int_{X}\Psi\left(\upsilon_{x},x\right)dm\left(x\right)=\int_{\mathcal{E}_{\varpi}^{G}\left(X\right)}\int_{X}\Psi\left(\nu,x\right)dm_{\nu}\left(x\right)d\widehat{\mu}\left(\nu\right).
\end{equation}
Indeed, since \(dm_{\nu}=\kappa d\nu\) and \(dm=\kappa d\mu\), this is just the ergodic decomposition of \(\mu\) applied to the nonnegative function \(\mathcal{E}_{\varpi}^{G}\left(X\right)\times X\to\mathbb{R}_{\geq 0}\), \(\left(\nu,x\right)\mapsto\Psi\left(\nu,x\right)\kappa\left(x\right)\). Define now
\[m_{x}:=m_{\upsilon_{x}}\quad\text{for }x\in X.\]
Then by \eqref{eq:hatiden}, for every Borel set \(A\subseteq X\),
\[m\left(A\right)=\int_{\mathcal{E}_{\varpi}^{G}\left(X\right)}m_{\nu}\left(A\right)d\widehat{\mu}\left(\nu\right)=\int_{X}m_{x}\left(A\right)d\mu\left(x\right).\]
It is clear that \(m_{x}\ll\upsilon_{x}\) for every \(x\), and it remains to show that \(m_{x}\) is \(G\)-invariant for \(m\)-a.e. \(x\in X\). Let us first show that for every fixed \(g\in G\), it holds that \(gm_{x}=m_{x}\) for \(m\)-a.e. \(x\in X\). Fix \(g\in G\). Given an arbitrary \(\Phi\in L^{\infty}\left(\widehat{\mu}\right)\), set \(\phi:=\Phi\circ\upsilon\in L^{\infty}\left(\mu\right)\), and since \(\upsilon\) is \(G\)-invariant, \(\phi\) is \(G\)-invariant. Let \(\mathcal{F}\subset L^{1}\left(m\right)\) be a countable collection of functions that together generate the \(\sigma\)-algebra of \(X\). By the \(G\)-invariance of \(m\) and \(\phi\), together with \eqref{eq:hatiden} applied to the integrable functions \(\left(\nu,x\right)\mapsto \Phi\left(\nu\right)f\left(g.x\right)\) and \(\left(\nu,x\right)\mapsto \Phi\left(\nu\right)f\left(x\right)\), for every \(f\in\mathcal{F}\) we get
\begin{align*}
\int_{\mathcal{E}_{\varpi}^{G}\left(X\right)}\Phi\left(\nu\right)gm_{\nu}\left(f\right)d\widehat{\mu}\left(\nu\right)
&=m\left(\phi\cdot gf\right)=m\left(g\phi\cdot gf\right)=m\left(\phi\cdot f\right)=\int_{\mathcal{E}_{\varpi}^{G}\left(X\right)}\Phi\left(\nu\right)m_{\nu}\left(f\right)d\widehat{\mu}\left(\nu\right).
\end{align*}
Since \(\Phi\) and \(f\) are arbitrary, it follows that \(gm_{\nu}=m_{\nu}\) for \(\widehat{\mu}\)-a.e. \(\nu\in\mathcal{E}_{\varpi}^{G}\left(X\right)\). Setting \(g\) free, define
\[E:=\{\left(g,\nu\right)\in G\times\mathcal{E}_{\varpi}^{G}\left(X\right):gm_{\nu}=m_{\nu}\},\]
and put \(E_{g}:=\{\nu\in\mathcal{E}_{\varpi}^{G}\left(X\right):\left(g,\nu\right)\in E\}\) for \(g\in G\) and \(E^{\nu}:=\{g\in G:\left(g,\nu\right)\in E\}\) for \(\nu\in\mathcal{E}_{\varpi}^{G}\left(X\right)\). For every \(g\in G\) we showed that \(\widehat{\mu}\left(E_{g}\right)=1\), so by Fubini's theorem there exists a \(\widehat{\mu}\)-conull set \(\mathcal{E}\subseteq\mathcal{E}_{\varpi}^{G}\left(X\right)\) such that \(m_{G}\left(G\setminus E^{\nu}\right)=0\) for every \(\nu\in\mathcal{E}\). However, note that \(E^{\nu}\) is a subgroup of \(G\) for every \(\nu\in\mathcal{E}_{\varpi}^{G}\left(X\right)\), and the only \(m_{G}\)-conull subgroup of \(G\) is \(G\) itself, hence \(E^{\nu}=G\) for all \(\nu\in\mathcal{E}\). Thus \(m_{\nu}\) is \(G\)-invariant for \(\widehat{\mu}\)-a.e. \(\nu\in\mathcal{E}_{\varpi}^{G}\left(X\right)\), or equivalently \(m_{x}\) is \(G\)-invariant for \(\mu\)- hence \(m\)-a.e. \(x\in X\).
\end{proof}

\subsection{Ergodicity of stationary actions}

We now specialize to ergodicity in stationary actions. Let \(\left(G,\theta\right)\) be a measured group and \(\left(X,\mu\right)\) a stationary \(\left(G,\theta\right)\)-space. It is a basic fact that ergodicity of \(\left(X,\mu\right)\) is equivalent to that \(\mu\) is an extreme point in the convex set of \(\theta\)-stationary probability measures on \(X\); this follows from \cite[Prop.~2.8]{benoist2016random} together with the following Fact~\ref{fct:invariance} (see also \cite[Cor.~2.7]{bader2006factor}).

\begin{fct}\label{fct:invariance}
For a stationary \(\left(G,\theta\right)\)-space \(\left(X,\mu\right)\), consider the Markov operator
\[T_{\theta}:L^{\infty}\left(\mu\right)\to L^{\infty}\left(\mu\right),\quad T_{\theta}f\left(x\right)=\int_{G}f\left(g.x\right)d\theta\left(g\right).\]
Then for every function \(f\in L^{\infty}\left(\mu\right)\) one has \(T_{\theta}f=f\) if and only if \(f\) is \(G\)-invariant.
\end{fct}

\begin{proof}
It is clear that if \(f\in L^{\infty}\left(\mu\right)\) is \(G\)-invariant then \(T_{\theta}f=f\). For the converse, let \(f\in L^{\infty}\left(\mu\right)\) be such that \(T_{\theta}f=f\). Define \(\Phi_{f}\in L^{\infty}\left(\theta\right)\) by
\[\Phi_{f}\left(g\right)=\left\Vert gf-f\right\Vert_{L^{2}\left(\mu\right)},\]
and we claim that \(\Phi_{f}\) is continuous. Let \(\pi\) be the unitary representation of Fact~\ref{fct:unifcont}(1), and write
\[\Phi_{f}\left(g\right)=\left\Vert\sqrt{\varpi_{g}}^{-1}\cdot\pi\left(g\right)f-f\right\Vert_{L^{2}\left(\mu\right)}\leq\left\Vert\sqrt{\varpi_{g}}^{-1}\cdot\pi\left(g\right)f-\pi\left(g\right)f\right\Vert_{L^{2}\left(\mu\right)}+\left\Vert\pi\left(g\right)f-f\right\Vert_{L^{2}\left(\mu\right)}.\]
The second term is continuous in \(g\) by Fact~\ref{fct:unifcont}(1), and the first term is bounded by
\[\left\Vert\sqrt{\varpi_{g}}^{-1}\cdot\sqrt{\varpi_{g}}-\sqrt{\varpi_{g}}\right\Vert_{L^{2}\left(\mu\right)}\cdot\left\Vert f\right\Vert_{L^{\infty}\left(\mu\right)}=\left\Vert\mathbf{1}-\pi_{g}\mathbf{1}\right\Vert_{L^{2}\left(\mu\right)}\cdot\left\Vert f\right\Vert_{L^{\infty}\left(\mu\right)},\]
which is continuous by Fact~\ref{fct:unifcont}(1). This shows that \(\Phi_{f}\) is continuous at \(e_{G}\). Now for all \(g,h\in G\) one has
\[\left|\Phi_{f}\left(hg\right)-\Phi_{f}\left(g\right)\right|=\big|\left\Vert hgf-f\right\Vert_{2}-\left\Vert gf-f\right\Vert_{L^{2}\left(\mu\right)}\big|\leq\left\Vert hgf-gf\right\Vert_{L^{2}\left(\mu\right)}=\Phi_{gf}\left(h\right),\]
and since \(\Phi_{gf}\) is continuous at \(e_{G}\) by the above argument, we get that \(\Phi_{f}\left(hg\right)\to\Phi_{f}\left(g\right)\) as \(h\to e_{G}\). This justifies the continuity of \(\Phi_{f}\). Using \(\theta\)-stationarity and Fubini's theorem, for every \(n\geq 1\) we have
\begin{align*}\int_{G}\Phi_{f}\left(g\right)^{2}d\theta^{\ast n}\left(g\right)
&=2\int_{X}f\left(x\right)^{2}d\mu\left(x\right)-2\int_{X}\int_{G}f\left(g.x\right)\cdot f\left(x\right)d\theta^{\ast n}\left(g\right)\\
&=2\int_{X}f\left(x\right)^{2}d\mu\left(x\right)-2\int_{X}T_{\theta^{\ast n}}f\left(x\right)\cdot f\left(x\right)d\mu\left(x\right)=0,
\end{align*}
hence \(\Phi_{f}\equiv 0\in L^{\infty}\left(\theta^{\ast n}\right)\). Since this holds for every \(n\geq 1\), by the admissibility of \(\theta\) it follows that \(\Phi_{f}=0\) on a dense subset of \(G\), and by continuity we conclude that \(\Phi_{f}\equiv0\).
\end{proof}

For a measured group \(\left(G,\theta\right)\), recall the probability space 
\[\left(\Omega,\mathbb{P}\right):=\left(G^{\mathbb{N}},\theta^{\otimes\mathbb{N}}\right),\]
and the shift transformation
\[\sigma:\Omega\to\Omega,\quad\sigma:\left(\omega_{0},\omega_{1},\omega_{2}\right)\mapsto\left(\omega_{1},\omega_{2},\dotsc\right).\]
For a stationary \(\left(G,\theta\right)\)-space \(\left(X,\mu\right)\), there is associated with the {\bf forward transformation}, \(\tau\), given by
\[\left(\Omega_{X},\mathbb{P}_{\mu}\right):=\left(\Omega\times X,\mathbb{P}\otimes\mu\right),\quad\tau\left(\omega,x\right)=\left(\sigma\omega,\omega_{0}.x\right).\]
In fact, the \(\theta\)-stationarity of \(\mu\) is fully encoded by the \(\mathbb{P}_{\mu}\)-invariance of \(\tau\). For details see \cite[\S 2.4]{benoist2016random}. Then by \cite[Prop.~2.14]{benoist2016random} for the Markov operator \(T_{\theta}\), combined with Fact~\ref{fct:invariance}, we have:

\begin{prop}\label{prop:ergodic}
A stationary \(\left(G,\theta\right)\)-space is ergodic (as a \(G\)-space) if and only if its associated forward transformation is ergodic (as a transformation).
\end{prop}

\subsection{Final proof of Theorem~\ref{thm:typeIII}}

The first part of Theorem~\ref{thm:typeIII} turns out to be a generalization of the fact that an ergodic probability preserving transformation admit no infinite {\sc a.c.i.m.}:

\begin{prop}\label{prop:ergodicacim}
Every ergodic stationary \(\left(G,\theta\right)\)-space admitting an {\sc a.c.i.m.} is necessarily a probability preserving \(G\)-space.
\end{prop}

\begin{proof}[Proof of Proposition~\ref{prop:ergodicacim}]
Let \(\left(X,\mu\right)\) be an ergodic stationary \(\left(G,\theta\right)\)-space and let \(m\) be an {\sc a.c.i.m.} Since the set \(\{dm/d\mu>0\}\) is \(G\)-invariant modulo \(\mu\), by ergodicity it is a \(\mu\)-conull set, and therefore \(m\sim\mu\). Let then \(\kappa:X\to\mathbb{R}_{>0}\) be a positive Borel function satisfying \(\kappa=d\mu/dm\) \(m\)- hence \(\mu\)-a.e. For the forward transformation \(\tau\) defined on \(\left(\Omega_{X},\mathbb{P}_{\mu}\right)=\left(\Omega\times X,\mathbb{P}\otimes\mu\right)\), consider the \(\sigma\)-finite measure \(\mathbb{P}_{m}:=\mathbb{P}\otimes m\). Note that \(\mathbb{P}_{m}\) is an {\sc a.c.i.m.} of \(\tau\) and, moreover, it is mutually absolutely continuous with \(\mathbb{P}_{\mu}\) and
\[\frac{d\mathbb{P}_{\mu}}{d\mathbb{P}_{m}}=\kappa\circ\pi_{X}\quad\mathbb{P}_{m}\text{- hence }\mathbb{P}_{\mu}\text{-a.e.,}\]
where \(\pi_{X}:\Omega_{X}\to X\) is the projection; indeed, the \(G\)-invariance of \(\mathbb{P}_{m}\) can be routinely verified on rectangular sets using the \(G\)-invariance of \(m\), so it follows by the monotone class theorem. Since both \(\mathbb{P}_{m}\) and \(\mathbb{P}_{\mu}\) are \(\tau\)-invariant, it follows that \(\kappa\circ\pi_{X}\) is a \(\tau\)-invariant function. By Proposition~\ref{prop:ergodic}, since \(\left(X,\mu\right)\) is \(G\)-ergodic, the forward transformation \(\tau\) is ergodic, and it follows that there exists \(c>0\) such that \(\kappa\circ\pi_{X}\equiv c\,\) \(\mathbb{P}_{m}\)- hence \(\mathbb{P}_{\mu}\)-a.e. By Fubini's theorem we deduce that \(\kappa\equiv c\,\) \(m\)- hence \(\mu\)-a.e. and therefore \(\mu\) is \(G\)-invariant, thus \(\left(X,\mu\right)\) is a probability preserving \(G\)-space.
\end{proof}

To deduce the second part of Theorem~\ref{thm:typeIII} we will use Lemma~\ref{lem:ergdecompacim}:

\begin{proof}[Final proof of Theorem~\ref{thm:typeIII}]
Let \(\left(X,\mu\right)\) be a stationary \(\left(G,\theta\right)\)-space with ergodic decomposition map \(\upsilon:X\to\mathcal{E}_{\varpi}^{G}\left(X\right)\).
Since \(\upsilon\) is \(G\)-invariant, we obtain the \(G\)-invariant measurable sets
\[X_{1}:=\{x\in X:\upsilon_{x}\text{ is }G\text{-invariant}\}\quad\text{and}\quad X_{\theta}:=X\backslash X_{1}.\]
Since \(X_{1}\in\mathcal{I}\left(G\right)\) while \(\upsilon_{x}\left(\cdot\right)=\mathbb{E}_{\mu}\left[\cdot\mid\mathcal{I}\left(G\right)\right]\left(x\right)\) for \(\mu\)-a.e. \(x\in X\), we obtain that
\[\mu\left(\cdot\cap X_{1}\right)=\int_{X_{1}}\upsilon_{x}\left(\cdot\right)d\mu\left(x\right)\text{ as measures on }X.\]
Therefore, by the defining property of \(X_{1}\), \(\mu\left(\cdot\cap X_{1}\right)\) is a \(G\)-invariant measure. Let the measurable set
\[E:=\big\{x\in X:\int_{G}\varpi_{g^{-1}}\left(x\right)d\theta\left(g\right)=1\big\}.\]
Then \(E\) is \(\mu\)-conull by the \(\theta\)-stationarity of \(\mu\), hence it is \(\upsilon_{x}\)-conull for \(\mu\)-a.e. \(x\in X\), so \(\upsilon_{x}\) is \(\theta\)-stationary for \(\mu\)-a.e. \(x\in X\). We finally show that \(\mu\left(\cdot\cap X_{\theta}\right)\) admits no {\sc a.c.i.m.} By the same reasoning as above,
\[\mu\left(\cdot\cap X_{\theta}\right)=\int_{X_{\theta}}\upsilon_{x}\left(\cdot\right)\text{ as measures on }X,\]
hence \(\mu\left(\cdot\cap X_{\theta}\right)\) is a \(\theta\)-stationary measure. Now if \(\mu\left(\cdot\cap X_{\theta}\right)\) has an {\sc a.c.i.m.} then by Lemma~\ref{lem:ergdecompacim} a positive measure of its ergodic components has an {\sc a.c.i.m.} However, these ergodic components are ergodic \(\theta\)-stationary measures which are not \(G\)-invariant, so by Proposition~\ref{prop:ergodicacim} they admit no {\sc a.c.i.m.}
\end{proof}

\section{Stationary actions of arbitrary type}\label{sct:type}

Recall that an ergodic nonsingular \(G\)-space \(\left(X,\mu\right)\) has its {\bf Krieger associated flow} defined as follows. Let \(\alpha=-\log\varpi:G\times X\to\mathbb{R}\) be the additive Radon--Nikodym cocycle of \(\left(X,\mu\right)\). Consider the Maharam extension of \(\left(X,\mu\right)\), which is the infinite measure preserving \(G\)-space
\[\left(X\times\mathbb{R},\mu\otimes\eta\right),\]
where \(\eta\) is given by \(d\eta\left(t\right)=e^{-t}dt\), with the action
\[g.\left(x,t\right)=\left(g.x,t+\alpha_{g}\left(x\right)\right).\]
This action of \(G\) commutes with the \(\mathbb{R}\)-action
\[s.\left(x,t\right)=\left(x,t+s\right),\]
and therefore \(\mathbb{R}\) acts on the space of ergodic components of the Maharam extension. This \(\mathbb{R}\)-action is called the {\bf Krieger associated flow}. When \(\left(X,\mu\right)\) is of type \(\mathrm{II}\) (either type \(\mathrm{II}_{1}\) or type \(\mathrm{II}_{\infty}\)), the cocycle \(\alpha\) is a coboundary and therefore the Maharam extension is isomorphic to the free transitive flow on the second coordinate. Then the significance of the Maharam extension is in type \(\mathrm{III}\) actions, where the Krieger associated flow enables one to further classify type \(\mathrm{III}\) actions into type \(\mathrm{III}_{\lambda}\) actions, \(0\leq\lambda\leq 1\) as follows. Type \(\mathrm{III}_{1}\) corresponds to the trivial flow on one point, in which case the Maharam extension is ergodic; Type \(\mathrm{III}_{\lambda}\) for \(0<\lambda<1\) corresponds to the periodic flow on \(\left[0,-\log\lambda\right)\); Type \(\mathrm{III}_{0}\) corresponds to all other properly ergodic (non-transitive) flow. In what follows, we adopt a construction introduced by Katznelson--Weiss~\cite[\S5]{katznelson1991classification} for \(G=\mathbb{Z}\) and extended by Vaes--Verjans \cite[\S3]{vaes2023orbit} for general \(G\), to show how other types can arise in stationary actions.

\smallskip

Let \(\left(G,\theta\right)\) be a measured group. Pick any probability preserving \(\mathbb{R}\)-space \(\left(Y,\nu\right)\). We will cook up a stationary \(\left(G,\theta\right)\)-space whose Krieger associated flow is \(\left(Y,\nu\right)\). Start with some stationary \(\left(G,\theta\right)\)-space \(\left(X,\mu\right)\) with additive Radon--Nikodym cocycle by \(\alpha=-\log\varpi:G\times X\to\mathbb{R}\). Consider the stationary \(\left(G,\theta\right)\)-space \(\left(X\times Y,\mu\otimes\nu\right)\) with the action
\[g.\left(x,y\right)=\left(g.x,\alpha_{g}\left(x\right).y\right).\]
Its Maharam extension is the measure preserving \(G\)-space \(\left(X\times Y\times\mathbb{R},\mu\otimes\nu\otimes\eta\right)\), with the action
\[g.\left(x,y,t\right)=\left(g.x,\alpha_{g}\left(x\right).y,t+\alpha_{g}\left(x\right)\right).\]
Denoting by \(L^{\infty}\left(\bullet\right)^{G}\) the space of \(G\)-invariant \(L^{\infty}\left(\bullet\right)\)-functions, define an injective map
\begin{equation}\label{eq:mapiso}
L^{\infty}\left(\nu\right)\to L^{\infty}\left(\mu\otimes\nu\otimes\eta\right)^{G},\quad f\mapsto\widetilde{f},\quad\widetilde{f}\left(x,y,t\right):=f\left(-t.y\right).
\end{equation}
(note \(\widetilde{f}\) is indeed \(G\)-invariant). The following is a special case of \cite[Proposition 3.4]{vaes2023orbit}.

\begin{prop}\label{prop:alltype}
Suppose \(\left(X,\mu\right)\) has type \(\mathrm{III}_{1}\) and \(\left(Y,\nu\right)\) is ergodic.
\begin{enumerate}
    \item The map \eqref{eq:mapiso} forms an \(\mathbb{R}\)-isomorphism 
    \(L^{\infty}\left(\nu\right)\cong L^{\infty}\left(\mu\otimes\nu\otimes\eta\right)^{G}\), where the latter space is equipped with the \(\mathbb{R}\)-action on the third coordinate.
    \item \(\left(X\times Y,\mu\otimes\nu\right)\) is an ergodic stationary \(\left(G,\theta\right)\)-space, whose Krieger associated flow is \(\left(Y,\nu\right)\).
\end{enumerate}
\end{prop}

\begin{proof}
For (1), the map \eqref{eq:mapiso} is evidently injective, and it is \(\mathbb{R}\)-equivariant since
\[s.\widetilde{f}\left(x,y,t\right)=\widetilde{f}\left(x,y,t-s\right)=f\left(\left(s-t\right).y\right)=s.f\left(-t.y\right)=\widetilde{s.f}\left(x,y,t\right).\]
For the surjectivity, let \(F\in L^{\infty}\left(\mu\otimes\nu\otimes\eta\right)^{G}\) be strictly \(G\)-invariant. For \(y\in Y\) let \(F_{y}\left(x,t\right):=F\left(x,t.y,t\right)\). Then \(F_{y}\) is \(G\)-invariant for the Maharam extension of \(\left(X,\mu\right)\), which is ergodic as \(\left(X,\mu\right)\) is of type \(\mathrm{III}_{1}\), hence \(F_{y}\left(x,t\right)\) is \(\mu\otimes\eta\)-a.e. a constant. Letting \(f\left(y\right)\) be this constant, then evidently \(F=\widetilde{f}\).

For (2), using part (1) and the ergodicity of \(\left(Y,\nu\right)\), there is an \(\mathbb{R}\)-isomorphisms
\[\mathbb{R}\cong L^{\infty}\left(\nu\right)^{\mathbb{R}}\cong\big(L^{\infty}\left(\mu\otimes\nu\otimes\eta\right)^{G}\big)^{\mathbb{R}}.\]
However, there is a \(G\)-equivariant embedding \(L^{\infty}\left(\mu\otimes\nu\right)^{G}\hookrightarrow\big(L^{\infty}\left(\mu\otimes\nu\otimes\eta\right)^{G}\big)^{\mathbb{R}}\) (via forgetting the third coordinate), so \(L^{\infty}\left(\mu\otimes\nu\right)^{G}\cong\mathbb{R}\) and \(\left(X\times Y,\mu\otimes\nu\right)\) is ergodic. Finally, by part (1), the space of \(G\)-ergodic components of the Maharam extension \(\left(X\times Y\times\mathbb{R},\mu\otimes\nu\otimes\eta\right)\) is \(\left(Y,\nu\right)\).
\end{proof}

There are many examples of stationary actions of type \(\mathrm{III}_{1}\) \cite{izumi2008ratio,bowen2014type,gekhtman2013stable}, so Proposition~\ref{prop:alltype} is not vacuous. However, in the above setting, unless \(\left(Y,\nu\right)\) is trivial, the stationary \(\left(G,\theta\right)\)-space \(\left(X\times Y,\mu\otimes\nu\right)\) is not a boundary and not even SAT (see \cite{furstenberg2010stationary}). This leads to the following famous problem, whose significance follows from the amenability of Poisson boundaries, where Krieger type plays a canonical role. Other than a few geometric setups \cite{izumi2008ratio,bowen2014type,gekhtman2013stable} it seems to be widely open (cf. \cite[\S4.4]{bowen2013pointwise}):

\begin{prob}
For what \(0\leq\lambda\leq 1\) type \(\mathrm{III}_{\lambda}\) occurs in the class of Poisson boundaries? In particular, does type \(\mathrm{III}_{0}\) ever occur in Poisson boundaries? What about other boundary actions?
\end{prob}

\section{Harnack's inequality in locally compact groups}\label{app:Harn}

\subsection{The harmonic majorant}

Let \(\left(G,\theta\right)\) be a measured group. In our study of Harnack's inequality, the central object is the normalized slice of the cone of positive \(\theta\)-harmonic functions on \(G\):
\[\mathbb{H}\left(G,\theta\right):=\{u:G\to\mathbb{R}_{>0}:u\ast\theta=u,\,u\left(e_{G}\right)=1\}.\]
Define the {\bf harmonic majorant} of \(\left(G,\theta\right)\) to be the function
\begin{equation}\label{eq:mjorant}
\mathfrak{m}_{\theta}:G\to\left[1,+\infty\right],\quad\mathfrak{m}_{\theta}\left(g\right):=\sup\big\{u\left(g\right):u\in\mathbb{H}\left(G,\theta\right)\big\}.
\end{equation}

We say that \(\left(G,\theta\right)\) satisfies the {\bf Harnack's inequality} if its harmonic majorant \(\mathfrak{m}_{\theta}\) is finite-valued. We note some basic properties of the harmonic majorant:

\begin{prop}\label{prop:mtheta}
For every measured group \(\left(G,\theta\right)\), the following properties hold:
\begin{enumerate}
    \item \(\mathfrak{m}_{\theta}\) is submultiplicative.
    \item If \(\mathfrak{m}_{\theta}\) is finite-valued, then it is locally bounded (i.e. bounded on compact sets).
    \item For every \(u\in\mathbb{H}\left(G,\theta\right)\) and every \(g,h\in G\),
    \[\mathfrak{m}_{\theta}\left(g^{-1}\right)^{-1}\leq\tfrac{u\left(hg\right)}{u\left(h\right)}\leq\mathfrak{m}_{\theta}\left(g\right).\]
    \item For every stationary \(\left(G,\theta\right)\)-space \(\left(X,\mu\right)\) and every \(g\in G\),
    \[\mathfrak{m}_{\theta}\left(g^{-1}\right)^{-1}\leq\tfrac{dg\mu}{d\mu}\left(x\right)\leq\mathfrak{m}_{\theta}\left(g\right)\text{ for }\mu\text{-a.e. }x\in X.\]
\end{enumerate}
\end{prop}

Before proving Proposition~\ref{prop:mtheta}, we note that local boundedness is automatic from submultiplicativity:

\begin{lem}\label{lem:submultip}
Every submultiplicative measurable nonnegative function on an lcsc group is locally bounded.
\end{lem}

\begin{proof}[Proof of Lemma~\ref{lem:submultip}]
Let \(f:G\to\left[0,+\infty\right)\) be submultiplicative measurable. Since the function  \(\widetilde{f}\left(g\right):=\max\left\{f\left(g\right),f\left(g^{-1}\right)\right\}\) is symmetric, while still submultiplicative measurable, and since \(f\left(g\right)\leq\widetilde{f}\left(g\right)\) for all \(g\in G\), by replacing \(f\) with \(\widetilde{f}\) we may assume that \(f\) is symmetric. Fix a compact identity neighborhood \(K\) in \(G\), and find \(c>0\) such that \(E:=K\cap\{f\leq c\}\) has positive Haar measure. By the Steinhaus--Weil theorem (see \cite[Thm.~2]{grosse1989extension}), the set \(EE^{-1}\) contains an identity neighborhood \(V\). We claim that \(f\) is bounded on \(V\); indeed, for \(g\in V\) write \(g=xy^{-1}\) with \(x,y\in E\), and by submultiplicativity
\[f\left(g\right)=f\left(xy^{-1}\right)\leq f\left(x\right)f\left(y^{-1}\right)=f\left(x\right)f\left(y\right)\leq c^{2}.\]
Given any compact set \(C\) in \(G\), there are \(g_{1},\dotsc,g_{M}\in G\) with \(C\subseteq\bigcup_{m=1}^{M}g_{m}V\). Then for every \(g\in C\), write \(g=g_{m^{\prime}}x\) with \(1\leq m^{\prime}\leq M\) and \(x\in V\), and by submultiplicativity
\[f\left(g\right)\leq f\left(g_{m^{\prime}}\right)f\left(x\right)\leq\left(\max\nolimits_{1\leq m\leq M}f\left(g_{m}\right)\right)\cdot c^{2}.\qedhere\]
\end{proof}

\begin{proof}[Proof of Proposition~\ref{prop:mtheta}]
We start by proving (3). Note that for every \(u\in\mathbb{H}\left(G,\theta\right)\) and \(h\in G\), the function \(g\mapsto u\left(hg\right)/u\left(h\right)\) belongs to \(\mathbb{H}\left(G,\theta\right)\), and therefore \(u\left(hg\right)/u\left(h\right)\leq\mathfrak{m}_{\theta}\left(g\right)\). Applying the same inequality with \(hg\) in place of \(h\) and \(g^{-1}\) in place of \(g\), we obtain that also \(\mathfrak{m}_{\theta}\left(g^{-1}\right)^{-1}\leq u\left(hg\right)/u\left(h\right)\). We now deduce (1). Fix \(h,g\in G\), and by part (3) we have just proved, for every \(u\in\mathbb{H}\left(G,\theta\right)\) we have
\[u\left(hg\right)=u\left(h\right)\cdot\tfrac{u\left(hg\right)}{u\left(h\right)}\leq\mathfrak{m}_{\theta}\left(h\right)\cdot\mathfrak{m}_{\theta}\left(g\right).\]
Taking the supremum over \(u\in\mathbb{H}\left(G,\theta\right)\) gives \(\mathfrak{m}_{\theta}\left(hg\right)\leq\mathfrak{m}_{\theta}\left(h\right)\cdot\mathfrak{m}_{\theta}\left(g\right)\). Now (2) follows by Lemma~\ref{lem:submultip}. We finally prove (4). For a measurable set \(A\subseteq X\) with \(\mu\left(A\right)>0\), let \(u_{A}\left(g\right):=\tfrac{1}{\mu\left(A\right)}\cdot g\mu\left(A\right)\in\mathbb{H}\left(G,\theta\right)\). Then from \(u_{A}\left(g\right)\leq\mathfrak{m}_{\theta}\left(g\right)\) we have
\[g\mu\left(A\right)=\mu\left(g^{-1}.A\right)=u_{A}\left(g\right)\cdot\mu\left(A\right)\leq\mathfrak{m}_{\theta}\left(g\right)\cdot\mu\left(A\right).\]
Since \(A\) is arbitrary, \(\tfrac{dg\mu}{d\mu}\left(x\right)\leq\mathfrak{m}_{\theta}\left(g\right)\) for \(\mu\)-a.e. \(x\in X\). Applying the same bound with \(g^{-1}\) in place of \(g\) and \(g.A\) in place of \(A\), we obtain
\[\mu\left(A\right)=g^{-1}\mu\left(g.A\right)\leq\mathfrak{m}_{\theta}\left(g^{-1}\right)\cdot\mu\left(g.A\right).\]
Then in the same way we obtain that \(\tfrac{dg\mu}{d\mu}\left(x\right)\geq\mathfrak{m}_{\theta}\left(g^{-1}\right)^{-1}\) for \(\mu\)-a.e. \(x\in X\).
\end{proof}

\subsection{A version of Harnack's inequality}

In the following we give a general sufficient condition for Harnack's inequality, and deduce the first part of Theorem~\ref{mthm:Harnack}. The key property is the following:

\begin{defn}
An element \(g\in G\) is called {\bf \(\theta\)-suitable}, if there are integers \(R\geq 1\) and \(k,k_{1},\dotsc,k_{R}\geq 0\), as well as constants \(\alpha_{1},\dotsc,\alpha_{R}>0\), such that
\[g\theta^{\ast k}\leq\sum\nolimits_{r=1}^{R}\alpha_{r}\cdot\theta^{\ast k_{r}},\]
as pointwise inequality of finite measures. Here, \(\theta^{\ast 0}=\delta_{\{e_{G}\}}\) denotes the Dirac measures on \(e_{G}\).
\end{defn}

\begin{lem}\label{lem:suitableHarnack}
For every \(\theta\)-suitable element \(g\in G\) one has \(\mathfrak{m}_{\theta}\left(g\right)<+\infty\). More explicitly, if \(g\in G\) is \(\theta\)-suitable with constants \(\alpha_{1},\dotsc,\alpha_{R}\), then
\[\mathfrak{m}_{\theta}\left(g\right)\leq\sum\nolimits_{r=1}^{R}\alpha_{r}.\]
\end{lem}

\begin{proof}[Proof of Lemma~\ref{lem:suitableHarnack}]
Let \(g\in G\) be \(\theta\)-suitable with \(k,k_{1},\dotsc,k_{R},\alpha_{1},\dotsc,\alpha_{R}\). Then for all \(u\in\mathbb{H}\left(G,\theta\right)\),
\[u\left(g\right)=\int_{G}u\left(gh\right)d\theta^{\ast k}\left(h\right)\leq\sum\nolimits_{r=1}^{R}\alpha_{r}\cdot\Big(\int_{G}u\left(h\right)d\theta^{\ast k_{r}}\left(h\right)\Big)=\sum\nolimits_{r=1}^{R}\alpha_{r}\cdot u\left(e_{G}\right)=\sum\nolimits_{r=1}^{R}\alpha_{r}.\qedhere\]
\end{proof}

The simplest case where all elements of \(G\) are \(\theta\)-suitable is in discrete countable groups (cf.~\cite[\S12.4]{lalley2023}):

\begin{exm}
If \(\left(G,\theta\right)\) is a discrete measured group, then for every \(g\in G\) there is \(n_{g}\geq 0\) with \(\theta^{\ast n_{g}}\left(g\right)>0\) by admissibility. Therefore, \(g\) becomes \(\theta\)-suitable by taking \(k=0\), \(R=1\), \(k_{1}=n_{g}\), \(\alpha_{1}=1/\theta^{\ast n_{g}}\left(g\right)\).
\end{exm}

Convolution semigroups on Lie groups is well-studied with delicate (elliptic, parabolic) Harnack's inequalities \cite{Varopoulos1992}. Still it is useful to look how the \(\theta\)-suitability condition arises naturally in such situations:

\begin{exm}
Let \(G\) be an lcsc group and let \(\left(\theta_{t}\right)_{t\geq0}\) be a convolution semigroup of probability measures,
\[\theta_{s}\ast\theta_{t}=\theta_{s+t},\quad 0<s<t,\]
with densities \(\varphi_{t}=d\theta_{t}/dm_{G}\). Assume that for every \(g\in G\) and \(s<t\) there is \(\alpha=\alpha\left(g,s,t\right)>0\) with
\[g\varphi_{s}\leq\alpha\cdot\varphi_{t}.\]
Then for every \(s>0\), every \(g\in G\) is \(\theta_{s}\)-suitable: taking \(t=2s\) gives \(g\varphi_{s}\leq\alpha\varphi_{2s}=\alpha\cdot\varphi_{s}^{\ast2}\), which implies \(g\theta_{s}\leq\alpha\cdot\theta_{s}^{\ast2}\).
Then the \(\theta_{s}\)-suitability condition holds with \(k=1\), \(R=1\), \(\alpha_{1}=\alpha\left(g,s,2s\right)\), \(k_{1}=2\).

A basic instance is \(G=\mathbb{R}^{d}\) with drifted Gaussian convolution semigroup \(\left(\theta_{t}\right)_{t\geq 0}\) with densities
\[\gamma_{t}\left(v\right):=\tfrac{1}{\left(4\pi t\right)^{d/2}}\cdot\exp\big(-\tfrac{\left\|v-tb\right\|^{2}}{4t}\big),\quad t>0,\]
for some \(0\neq b\in\mathbb{R}^{d}\). For every \(g\in\mathbb{R}^{d}\) and \(0<s<t\) one has
\[\tfrac{\gamma_{s}\left(v-g\right)}{\gamma_{t}\left(v\right)}=\left(t/s\right)^{d/2}\cdot\exp\left(-\tfrac{\left\|v-g-sb\right\|^{2}}{4s}+\tfrac{\left\|v-tb\right\|^{2}}{4t}\right),\]
and the exponent is bounded from above by \(\left\|g-\left(t-s\right)b\right\|^{2}/4\left(t-s\right)\). Therefore,
\[g\gamma_{s}\left(v\right)=\gamma_{s}\left(v-g\right)\leq \alpha\left(g,s,t\right)\cdot\gamma_{t}\left(v\right),\]
where
\[\alpha\left(g,s,t\right)=\left(t/s\right)^{d/2}\cdot\exp\left(\tfrac{\left\|g-\left(t-s\right)b\right\|^{2}}{4\left(t-s\right)}\right).\]
\end{exm}

\begin{proof}[Proof of Theorem~\ref{mthm:Harnack} (first part)]
Using Lemma~\ref{lem:suitableHarnack}, it suffices to show that every element of \(G\) is \(\theta\)-suitable. By the assumption, there is \(1<p\leq+\infty\) so that \(\varphi:=d\theta/dm_{G}\in L_{\mathrm{c}}^{p}\left(G\right)\). If \(p=+\infty\), then \(\varphi\in L^{2}_{\mathrm{c}}\left(G\right)\) and hence \(\varphi^{\ast 2}\in C_{\mathrm{c}}\left(G\right)\). If \(1<p<\infty\), then by iterating Young's convolution inequality  \cite[Lemma~2.1]{klein1978sharp} (cf. \cite[Cor.~4.4]{benoist2016convolution}) there is \(n_{p}\geq 1\) such that \(\varphi^{\ast n_{p}}\in L^{2}\left(G\right)\). Since supports remain compact under convolution, we have \(\varphi^{\ast n_{p}}\in L^{2}_{\mathrm{c}}\left(G\right)\), and hence \(\varphi^{\ast\left(2n_{p}\right)}=\varphi^{\ast n_{p}}\ast\varphi^{\ast n_{p}}\in C_{\mathrm{c}}\left(G\right)\). Set \(N_{p}:=2\) if \(p=\infty\) and \(N_{p}:=2n_{p}\) if \(1<p<\infty\), we obtain \(\varphi^{\ast N_{p}}\in C_{\mathrm{c}}\left(G\right)\). Let then the open set \(O:=\left\{\varphi^{\ast N_{p}}>0\right\}\) and the compact set \(K:=\mathrm{supp}\left(\theta^{\ast N_{p}}\right)=\overline{O}\). We now claim that for every \(x\in G\) there is \(n_{x}\geq 1\) and a neighborhood \(V_{x}\ni x\), such that
\begin{equation}\label{eq:posit}
c_{x}:=\inf\varphi^{\ast\left(n_{x}+N_{p}\right)}\mid_{V_{x}}>0.
\end{equation}
Indeed, since \(xO\) is open, by the admissibility of \(\theta\) there is \(n_{x}\geq 1\) such that \(\theta^{\ast n_{x}}\left(xO\right)>0\). Using that \(\varphi^{\ast N_{p}}\left(x^{-1}h\right)>0\) for \(h\in xO\), we get
\[\varphi^{\ast\left(n_{x}+N_{p}\right)}\left(x\right)=\int_{G}\varphi^{\ast N_{p}}\left(x^{-1}h\right)d\theta^{\ast n_{x}}\left(h\right)>0.\]
Since \(\varphi^{\ast\left(n_{x}+N_{p}\right)}=\varphi^{\ast n_{x}}\ast\varphi^{\ast N_{p}}\) is continuous, \eqref{eq:posit} follows. Fix \(g\in G\). By the compactness of \(gK\), there are \(x_{1},\dotsc,x_{R}\in gK\) so that the corresponding \(V_{x_{1}},\dotsc,V_{x_{R}}\) cover \(gK\). Then for every \(y\in G\), either \(y\notin gK\) and then
\[g\varphi^{\ast N_{p}}\left(y\right)=\varphi^{\ast N_{p}}\left(g^{-1}y\right)=0,\]
or otherwise \(y\in gK\), say \(y\in V_{x_{r}}\) for some \(1\leq r\leq R\), and then
\[g\varphi^{\ast N_{p}}\left(y\right)\leq \left\Vert\varphi^{\ast N_{p}}\right\Vert_{L^{\infty}\left(G\right)}
\leq\left\Vert\varphi^{\ast N_{p}}\right\Vert_{L^{\infty}\left(G\right)}\cdot c_{x_{r}}^{-1}\cdot\varphi^{\ast\left(n_{x_{r}}+N_{p}\right)}\left(y\right),\]
with \(c_{x_{r}}\) defined by \eqref{eq:posit}. Letting \(\alpha_{r}:=\left\Vert\varphi^{\ast N_{p}}\right\Vert_{L^{\infty}\left(G\right)}\cdot c_{x_{r}}^{-1}\) for \(1\leq r\leq R\), we found that for every \(y\in G\),
\[g\varphi^{\ast N_{p}}\left(y\right)\leq \sum\nolimits_{r=1}^{R}\alpha_{r}\cdot \varphi^{\ast\left(n_{x_{r}}+N_{p}\right)}\left(y\right),\]
which implies the inequality of measures
\[g\theta^{\ast N_{p}}\leq \sum\nolimits_{r=1}^{R}\alpha_{r}\cdot \theta^{\ast\left(n_{x_{r}}+N_{p}\right)}.\]
Thus, \(g\) is \(\theta\)-suitable with \(k=N_{p}\), \(k_{r}=n_{x_{r}}+N_{p}\) for \(1\leq r\leq R\), and the above \(\alpha_{1},\dotsc,\alpha_{R}\).
\end{proof}

\subsection{The failure of Harnack's inequality: A non-SAT\(^{\ast}\) Poisson boundary}

The classical (planar) Harnack inequality \cite[\S19]{harnack1887grundlagen} was formulated for local operators, which in the setting of lcsc groups corresponds to compactly supported measures. For nonlocal operators, Harnack's classical inequality holds only under extra assumptions, and its general failure was demonstrated by Kassmann \cite{Kassmann2007,Kassmann2011}.

\smallskip

In Kassmann’s counterexample, the obstruction arises from nonlocal influence of the complement of the domain on which the Harnack's inequality is studied, and the contribution of the negative part of the harmonic function outside the domain. We use a fundamentally different mechanism for the failure of Harnack's inequality for measured groups, and we thus can show that it fails already for everywhere positive harmonic functions.

\smallskip

Recall the {\bf SAT} property (Strongly Approximately Transitive) defined by Jaworski \cite{jaworski1995strong} for a nonsingular \(G\)-space \(\left(X,\mu\right)\), by which for every Borel set \(A\) in \(X\) with \(\mu\left(A\right)>0\), one has \(\sup_{g\in G}g\mu\left(A\right)=1\). Every Poisson boundary, and in fact every boundary action (so called \emph{proximal}), has the SAT property; this was proved by Jaworski for discrete countable groups, and by Furstenberg--Glasner for lcsc groups \cite[\S8]{furstenberg2010stationary}. A slightly stronger version of the SAT property was introduced by Kaimanovich \cite{Kaimanovich2002SAT} under the name {\bf SAT\(^{\ast}\)} property, by which one requires also that \(\big\Vert\frac{dg\mu}{d\mu}\big\Vert_{L^{\infty}\left(\mu\right)}<+\infty\) for all \(g\in G\).

\smallskip

In the following we prove the second part of Theorem~\ref{mthm:Harnack}, by constructing a random walk on the real affine group, whose Poisson boundary fails to be SAT\(^{\ast}\). This will be done by choosing a density of the increments distribution that oscillates locally so it is not uniformly comparable to its translates, which results in a Radon--Nikodym cocycle unbounded near the origin.

\smallskip

Fix the group of affine transformations of \(\mathbb{R}\) with the operation of composition:
\[G:=\mathrm{Aff}\left(\mathbb{R}\right)
=\left\{x\mapsto ax+b:a\in\mathbb{R}_{>0},\ b\in\mathbb{R}\right\}.\]
The standard left Haar measure of \(G\) is given by
\[dm_{G}\left(a,b\right)=a^{-2}dadb,\]
where the pair \(\left(a,b\right)\) stands for \(x\mapsto ax+b\), and \(da\) and \(db\) stand for the usual Lebesgue measure.

\begin{thm}\label{thm:cntexm}
There exists an admissible probability measure \(\theta\) on \(\mathrm{Aff}\left(\mathbb{R}\right)\) with Poisson boundary \(\left(\mathbb{R},\mu\right)\), such that for every element \(g_{t}:=\left(x\mapsto x+t\right)\in\mathrm{Aff}\left(\mathbb{R}\right)\) with \(t\neq0\), one has
\[\big\Vert\tfrac{dg_{t}\mu}{d\mu}\big\Vert_{L^{\infty}\left(\mathbb{R}\right)}=+\infty.\]
Therefore, this Poisson boundary is SAT but not SAT\(^{\ast}\). In particular, \(\left(\mathrm{Aff}\left(\mathbb{R}\right),\theta\right)\) does not satisfy Harnack's inequality since \(\mathfrak{m}_{\theta}\left(g_{t}\right)=+\infty\) for every \(t\neq0\).
\end{thm}

We start by recalling some basics of random walks on \(G=\mathrm{Aff}\left(\mathbb{R}\right)\). Let \(A\) and \(B\) be independent random variables with values in \(\mathbb{R}_{>0}\) and \(\mathbb{R}\), respectively. Define the \(G\)-valued random variable
\[\Theta:=\left(x\mapsto Ax+B\right).\]
The random walk whose (independent) increments distributed as \(\Theta\), takes the form
\[S_{N}:=\big(x\mapsto Q_{N}x+R_{N}\big),\quad N=1,2,3,\dotsc,\]
where for independent two i.i.d sequences \(A_{1},A_{2},\dotsc\sim A\) and \(B_{1},B_{2},\dotsc\sim B\), we write
\[Q_{0}\equiv 1,\quad Q_{N}:=\prod\nolimits_{n=1}^{N}A_{n}\quad\text{and}\quad R_{N}:=\sum\nolimits_{n=1}^{N}Q_{n-1}B_{n},\quad N\geq 1.\]

\begin{lem}\label{lem:ABconv}
Consider the following properties of \(A\) and \(B\):
\begin{equation}\label{eq:ABmom}
-\infty<\mathbb{E}\left[\log A\right]<0\quad\text{and}\quad\mathbb{E}\left[\log^{+}\left|B\right|\right]<+\infty.\qquad(\log^{+}\left|b\right|:=\max\{0,\log\left|b\right|\}).
\end{equation}
Then under \eqref{eq:ABmom}, the limit \(R_{\infty}:=\lim\nolimits_{N\to\infty}R_{N}\) exists and is finite almost surely.
\end{lem}

\begin{proof}[Proof of Lemma~\ref{lem:ABconv}]
By the strong law of large numbers, almost surely
\[\quad \tfrac{1}{n}\log Q_{n}=\tfrac{1}{n}\sum\nolimits_{k=1}^{n}\log A_{k}\xrightarrow[n\to\infty]{}\mathbb{E}\left[\log A\right]<0,\]
so there are \(\alpha>0\) and a random integer \(T_{\alpha}\) such that \(Q_{n}\leq e^{-\alpha n}\) for all \(n\geq T_{\alpha}\). Fix \(0<\beta<\alpha\). Then by the tail integral formula we have
\begin{align*}
\sum\nolimits_{n=1}^{\infty}\mathbb{P}\left(\log^{+}\left|B_{n}\right|>\beta n\right)
&\leq\beta^{-1}\sum\nolimits_{n=1}^{\infty}\int_{\beta\left(n-1\right)}^{\beta n}\mathbb{P}\left(\log^{+}\left|B_{n}\right|>t\right)dt\\
&=\beta^{-1}\int_{0}^{\infty}\mathbb{P}\left(\log^{+}\left|B\right|>t\right)dt=\beta^{-1}\mathbb{E}\left[\log^{+}\left|B\right|\right]<+\infty.
\end{align*}
Therefore, by the Borel--Cantelli lemma there is a random integer \(T_{\beta}\) such that \(\left|B_{n}\right|\leq e^{\beta n}\) for all \(n\geq T_{\beta}\). Set \(T=\max\left\{T_{\alpha},T_{\beta}\right\}\). Then for all \(n\geq T\) we have
\[\quad \left|Q_{n-1}B_{n}\right|\leq e^{-\alpha\left(n-1\right)}e^{\beta n}=e^{\alpha-\left(\alpha-\beta\right)n}.\]
Since \(\alpha>\beta\), it follows that \(\sum\nolimits_{n=1}^{\infty}Q_{n-1}B_{n}\) converges (absolutely) almost surely.
\end{proof}

\begin{lem}\label{lem:ABgen}
Consider the following properties of \(A\) and \(B\):
\begin{equation}\label{eq:ABsupp}
\begin{aligned}
& A \text{ has a density positive on }
\left(1-\delta_{-},1+\delta_{+}\right)
\text{ for some }0<\delta_{-},\delta_{+}<1,\text{ and}\\
& B \text{ has a density positive Lebesgue-a.e. on }\mathbb{R}.
\end{aligned}
\end{equation}
Then under \eqref{eq:ABsupp}, \(\left(G,\theta\right):=\left(\mathrm{Aff}\left(\mathbb{R}\right),\mathrm{Law}\left(\Theta\right)\right)\) forms a measured group.
\end{lem}

\begin{proof}[Proof of Lemma~\ref{lem:ABgen}]
Write \(f_{A}\) and \(f_{B}\) for the densities of \(A\) and \(B\). Since \(A\) and \(B\) are independent,
\[d\theta\left(a,b\right)=f_{A}\left(a\right)f_{B}\left(b\right)dadb.\]
Then with the standard left Haar measure of \(G\) we get
\[d\theta\left(a,b\right)=a^2f_{A}\left(a\right)f_{B}\left(b\right)dm_{G}\left(a,b\right),\text{ hence }\theta\ll m_{G}.\]
Moreover, for \(a\in\mathrm{supp}\left(A\right)\supseteq\left(1-\delta_{-},1+\delta_{+}\right)\) we have \(\left(a,0\right)\in\mathrm{supp}\left(\theta\right)\), and for \(b\in\mathrm{supp}\left(B\right)=\mathbb{R}\) we have \(\left(1,b\right)\in\mathrm{supp}\left(\theta\right)\). Since \(\left(1,b\right)\left(a,0\right)=\left(a,b\right)\) and the multiplicative semigroup generated by
\(\left(1-\delta_{-},1+\delta_{+}\right)\) is \(\mathbb{R}_{>0}\), the closed semigroup generated by \(\mathrm{supp}\left(\theta\right)\) is \(G\).
\end{proof}

In the following we appeal to a theorem of \'{E}lie \cite[Th\'eor\`eme p.~38]{Elie1978}, \cite[Prop.~2.2]{Elie1983} (see \cite[\S3.F]{cartwright1994random}), by which in the contracting regime \(\mathbb{E}\left[\log A\right]<0\), the Poisson boundary of \(\left(G,\theta\right)=\left(\mathrm{Aff}\left(\mathbb{R}\right),\mathrm{Law}\left(\Theta\right)\right)\) is the unique \(\theta\)-stationary measure on \(\mathbb{R}\) (considered with the tautological action of \(G\)).

\begin{lem}\label{lem:Rinfstat}
Let \(A\) and \(B\) be such that \eqref{eq:ABmom} and \eqref{eq:ABsupp} hold. Set \(\mu:=\mathrm{Law}\left(R_{\infty}\right)\) with \(R_{\infty}\) as in Lemma~\ref{lem:ABconv}, and set also \(\Theta=\left(x\mapsto Ax+B\right)\) and \(\theta:=\mathrm{Law}\left(\Theta\right)\). Then the following hold.
\begin{enumerate}
    \item \(\left(\mathbb{R},\mu\right)\) is a stationary \(\left(G,\theta\right)\)-space; hence, by \'{E}lie's theorem it is the Poisson boundary of \(\left(G,\theta\right)\).
    \item \(\mu\) admits a Lebesgue-a.e. positive density \(\rho_{\infty}\), and for every \(g=\left(a,b\right)\in G\) one has
    \[\tfrac{dg\mu}{d\mu}\left(x\right)=\tfrac{1}{a}\cdot\tfrac{\rho_{\infty}\left(a^{-1}\left(x-b\right)\right)}{\rho_{\infty}\left(x\right)}\quad\text{for Lebesgue-a.e. }x\in\mathbb{R}.\]
\end{enumerate}
\end{lem}

\begin{proof}[Proof of Lemma~\ref{lem:Rinfstat}]
Let
\[R_{\infty}^{\prime}:=B_{2}+\sum\nolimits_{n=3}^{\infty}\big(\prod\nolimits_{k=2}^{n-1}A_{k}\big)B_{n}.\]
Then \(R_{\infty}^{\prime}\overset{\mathrm d}=R_{\infty}\), and \(R_{\infty}^{\prime}\) is independent of \(A_{1},B_{1}\). Hence
\[R_{\infty}=A_{1}R_{\infty}^{\prime}+B_{1}\overset{\mathrm d}=AR_{\infty}+B=\Theta\left(R_{\infty}\right),\]
and thus \(\mu\) is \(\theta\)-stationary. Now set \(U:=A_{1}R_{\infty}^{\prime}\). Since \(R_{\infty}^{\prime}\) is independent of \(A_{1},B_{1}\), it follows that \(U\) is independent of \(B_{1}\). Let \(f_{B}\) be the (Lebesgue-a.e. positive) density of \(B\), and let \(\nu:=\mathrm{Law}\left(U\right)\). Since \(R_{\infty}=B_{1}+U\), it follows by Fubini's theorem that \(\mu=\mathrm{Law}\left(R_{\infty}\right)\) is absolutely continuous with density
\[\rho_{\infty}\left(x\right)=\int_{\mathbb{R}}f_{B}\left(x-y\right)d\nu\left(y\right)>0.\]
For every test function \(f\in L^{\infty}\left(\mu\right)\), by a linear change of variable one has
\[\int_{\mathbb{R}}f\left(g.x\right)\rho_{\infty}\left(x\right)dx=\int_{\mathbb{R}}f\left(ax+b\right)\rho_{\infty}\left(x\right)dx=\int_{\mathbb{R}}f\left(x\right)\cdot\tfrac{1}{a}\rho_{\infty}\left(a^{-1}\left(x-b\right)\right)dx,\]
and this shows that
\[\tfrac{dg\mu}{d\mathrm{Lebesgue}}\left(x\right)=\tfrac{1}{a}\rho_{\infty}\left(a^{-1}\left(x-b\right)\right).\]
Dividing by \(\rho_{\infty}=d\mu/d\mathrm{Lebesgue}\) gives the desired formula.
\end{proof}

\begin{proof}[Proof of Theorem~\ref{thm:cntexm}]
Fix parameters \(0<\beta<\alpha<1\) and \(0<\delta<1/4\). Let \(A\) and \(B\) be independent random variables defined as follows. First, let \(A\) admit the density with respect to Lebesgue measure
\[f_{A}\left(a\right)=c_{A}\cdot\left(Ma^{\beta-1}\mathbf{1}_{\left(0,1/2\right)}\left(a\right)+\mathbf{1}_{\left(1-\delta,1+\delta\right)}\left(a\right)\right),\]
where \(c_{A}>0\) is a normalizing constant, and \(M>0\) is chosen large enough so that \(\mathbb{E}\left[\log A\right]<0\). Since
\[\int_{0}^{1/2}a^{\beta-1}\left|\log a\right|\,da<\infty,\]
we also have \(\mathbb{E}\left[\log A\right]>-\infty\). Next, let \(B\) admit the density with respect to Lebesgue measure
\[f_{B}\left(b\right)=c_{B}\cdot\left|b\right|^{-\alpha}e^{-b^{2}},\quad b\neq0,\]
where \(c_{B}>0\) is a normalizing constant (ignore the value at \(b=0\)). Then \(f_{B}>0\) Lebesgue-a.e. and \(\mathbb{E}\left[\log^{+}\left|B\right|\right]<\infty\). Therefore \eqref{eq:ABmom} and \eqref{eq:ABsupp} hold. Let \(\Theta=\left(x\mapsto Ax+B\right)\), and we will show that the conclusion of the theorem holds with \(\theta:=\mathrm{Law}\left(\Theta\right)\). Let \(R_{\infty}\) and \(\rho_{\infty}\) be as in Lemma~\ref{lem:Rinfstat}. Set
\[R_{\infty}^{\prime}:=B_{2}+\sum\nolimits_{n=3}^{\infty}\big(\prod\nolimits_{k=2}^{n-1}A_{k}\big)B_{n},\quad U:=A_{1}R_{\infty}^{\prime}.\]
Then \(R_{\infty}^{\prime}\overset{\mathrm d}=R_{\infty}\), \(R_{\infty}^{\prime}\) is independent of \(A_{1},B_{1}\), and
\[R_{\infty}\overset{\mathrm d}=B_{1}+U.\]
Since \(R_{\infty}^{\prime}\) has density \(\rho_{\infty}\) and it is independent of \(A_{1}\), the usual product density formula gives that
\[\rho_{U}\left(y\right)=\int_{\mathbb{R}}\tfrac{1}{\left|x\right|}f_{A}\left(y/x\right)\rho_{\infty}\left(x\right)dx.\]
Since \(R_{\infty}\overset{\mathrm d}=B_{1}+U\) and \(B_{1}\) is independent of \(U\), the density of \(R_{\infty}\) is given by the convolution
\[\rho_{\infty}\left(x\right)=\int_{\mathbb{R}}f_{B}\left(x-y\right)\rho_{U}\left(y\right)dy.\]
We will now use these formulas for \(\rho_{U}\) and \(\rho_{\infty}\) to estimate \(\rho_{\infty}\) near \(0\) and away from \(0\).
\begin{enumerate}
    \item We will show that \(\rho_{\infty}\left(x\right)\to+\infty\) as \(x\ssearrow0\). Since \(\rho_{\infty}>0\) Lebesgue-a.e. we have
    \[c_{0}:=\int_{1}^{2}x^{-\beta}\rho_{\infty}\left(x\right)dx>0.\]
    For every \(1<x<2\), if \(0<y<1/4\) then \(0<y/x<1/2\), hence
    \[\rho_{U}\left(y\right)\ge c_{A}My^{\beta-1}\cdot\int_{1}^{2}x^{-\beta}\rho_{\infty}\left(x\right)dx=c_{0}c_{A}My^{\beta-1}.\]
    For \(0<x<1/8\) and \(0<y<x/2\), one has \(x/2<x-y<x\), and therefore
    \[f_{B}\left(x-y\right)=c_{B}\left(x-y\right)^{-\alpha}e^{-\left(x-y\right)^{2}}\ge c_{B}x^{-\alpha}e^{-1}.\]
    Combining both bounds we obtain
    \[\rho_{\infty}\left(x\right)\ge \int_{0}^{x/2}f_{B}\left(x-y\right)\rho_{U}\left(y\right)dy\geq c_{B}x^{-\alpha}e^{-1}c_{0}c_{A}M\cdot\int_{0}^{x/2}y^{\beta-1}dy=\frac{c_{0}c_{A}c_{B}M}{e\beta 2^{\beta}}\cdot x^{\beta-\alpha}.\]
    Since \(\beta<\alpha\), the desired limit follows.
    \item We will now show that for every \(t\neq 0\),
    \[C_{t}:=\sup_{0<x<\left|t\right|/8}\rho_{\infty}\left(t+x\right)<+\infty.\]  
    For every \(y\neq0\), from the definition of \(f_A\) we may write
    \[\rho_{U}\left(y\right)\leq c_{A}M\left|y\right|^{\beta-1}\cdot\int_{\{0<y/x<1/2\}}\left|x\right|^{-\beta}\rho_{\infty}\left(x\right)dx+c_{A}\cdot\int_{\{1-\delta<y/x<1+\delta\}}\left|x\right|^{-1}\rho_{\infty}\left(x\right)dx.\]
    On these two regions one has \(\left|x\right|\ge 2\left|y\right|\) and \(\left|x\right|\ge \left|y\right|/\left(1+\delta\right)\), respectively, hence
    \[\rho_{U}\left(y\right)\le c_{A}M2^{-\beta}\left|y\right|^{-1}+c_{A}\left(1+\delta\right)\left|y\right|^{-1}.\]
    Now fix \(t\neq0\) and set \(\eta:=\left|t\right|/8\). For every \(0<x<\eta\), split
    \[\rho_{\infty}\left(t+x\right)=\int_{\left|b\right|<\eta}f_{B}\left(b\right)\rho_{U}\left(t+x-b\right)db+\int_{\left|b\right|\ge\eta}f_{B}\left(b\right)\rho_{U}\left(t+x-b\right)db.\]
    For the first term, when \(\left|b\right|<\eta\) we have
    \[\left|t+x-b\right|\ge \left|t\right|-\left|x\right|-\left|b\right|>\left|t\right|/2,\]
    hence
    \[\rho_{U}\left(t+x-b\right)\le 2c_{A}M2^{-\beta}\left|t\right|^{-1}+2c_{A}\left(1+\delta\right)\left|t\right|^{-1}.\]
    Since \(\int_{\mathbb{R}}f_{B}\left(b\right)db=1\), the first term is bounded by the same constant. For the second term, since \(f_{B}\) is bounded on \(\{\left|b\right|\ge\eta\}\) and \(\int_{\mathbb{R}}\rho_{U}\left(y\right)dy=1\), it is bounded by \(\left\|f_{B}\mathbf{1}_{\{\left|b\right|\ge\eta\}}\right\|_{L^{\infty}\left(\mathbb{R}\right)}\).
\end{enumerate}
Finally, let \(t\neq 0\) be arbitrary. Then for \(g_{t}:=\left(1,t\right)\), using the formula in Lemma~\ref{lem:Rinfstat} we obtain
\[\tfrac{dg_{t}\mu}{d\mu}\left(t+x\right)=\tfrac{\rho_{\infty}\left(x\right)}{\rho_{\infty}\left(t+x\right)}\geq\rho_{\infty}\left(x\right)/C_{t}\to+\infty\text{ as }x\ssearrow0,\text{ and thus }\big\Vert\tfrac{dg_{t}\mu}{d\mu}\big\Vert_{L^{\infty}\left(\mathbb{R}\right)}=+\infty\qedhere.\]
\end{proof}

\section{Compact Radon--Nikodym models}\label{sct:compRN}

Topological models of measurable systems are a fundamental tool used widely across measurable dynamics, and in particular in stationary dynamics; see e.g. \cite[\S8]{furstenberg2010stationary}, \cite[\S3]{furstenberg2013recurrence}, \cite[Thm.~2.1]{bader2006factor}. The goal is to view a given measurable dynamical system as an invariant subset of a compact space with continuous dynamics. It is of particular interest, especially in descriptive set theoretic aspects of Borel dynamics, to find for a given group a universal compact metric space in which all topological models of its measurable actions are realized; see e.g. \cite{becker1993borel,becker1996descriptive,Hjorth1999}. The existence of such a universal model for actions of lcsc groups is a classical result due to Mackey \cite[\S2]{mackey1962point} and Varadarajan \cite[Thm.~5.7]{varadarajan1968geometry} (see \cite[Thm.~2.1.19]{zimmer2013ergodic}).

\begin{thm}[Mackey--Varadarajan's Compact model theorem]\label{thm:Vara}
Every lcsc group \(G\) admits a compact \(G\)-space \(\mathbb{Y}=\mathbb{Y}\left(G\right)\) (i.e. a compact metric space with a continuous action of \(G\)), such that every Borel \(G\)-space \(X\) admits a {\bf compact model} in \(\mathbb{Y}\), that is, a \(G\)-equivariant measurable injective map \(\kappa:X\to\mathbb{Y}\).
\end{thm}

In the presence of a quasi-invariant measure on \(X\), it is natural to ask for a compact model in which not only the action is continuous, but also the Radon--Nikodym cocycle can be taken continuous. In Poisson boundaries finding such a model is a classical topic going back to the seminal work of Furstenberg on semisimple Lie groups \cite{ furstenberg1963poisson,furstenberg1963noncomm}, and its many extensions; see e.g. \cite[\S2.8]{kaimanovich1996boundaries}, \cite{kaimanovich1995poisson}, \cite[\S4]{benoist2016random}, \cite[\S3.10]{furman2002random}). Here we show that in many common situations this can be done for all stationary actions. Let us give a precise definition of the object we seek for.

\begin{defn}
Let \(\left(G,\theta\right)\) be a measured group and \(\left(X,\mu\right)\) a stationary \(\left(G,\theta\right)\)-space. A {\bf compact Radon--Nikodym model} for \(\left(X,\mu\right)\) is a stationary \(\left(G,\theta\right)\)-space \(\big(\widehat{X},\widehat{\mu}\big)\) and a cocycle \(\eta:G\times\widehat{X}\to\mathbb{R}_{>0}\), such that the following hold:
\begin{enumerate}
    \item \(\widehat{X}\) is a compact \(G\)-space: a compact metric space with a continuous action of \(G\).
    \item \(\eta\) is a continuous version of the Radon--Nikodym cocycle of \(\big(\widehat{X},\widehat{\mu}\big)\) (recall Definition~\ref{dfn:strversion}).
    \item \(\eta\) is \(\theta\)-harmonic: for every \(x\in\widehat{X}\), the function \(G\to\mathbb{R}_{>0}\), \(g\mapsto\eta_{g^{-1}}\left(\widehat{x}\right)\), is \(\theta\)-harmonic.
    \item \(\left(X,\mu\right)\) is measurably isomorphic to \(\big(\widehat{X},\widehat{\mu}\big)\).\footnote{There is a \(\mu\)-conull set \(X_{o}\subseteq X\) and an injective \(G\)-equivariant map \(X_{o}\to\widehat{X}\) that pushes \(\mu\) to \(\widehat{\mu}\).}
\end{enumerate}
\end{defn}

We further seek for a universal compact Radon--Nikodym model:

\begin{defn}
Let \(\left(G,\theta\right)\) be a measured group. A {\bf universal compact Radon--Nikodym model} for \(\left(G,\theta\right)\) is a compact \(G\)-space \(\mathbb{G}\) with a \(\theta\)-harmonic continuous cocycle \(\mathfrak{g}:G\times\mathbb{G}\to\mathbb{R}_{>0}\), such that every stationary \(\left(G,\theta\right)\)-space \(\left(X,\mu\right)\) admits a compact Radon--Nikodym model \(\big(\widehat{X},\widehat{\mu}\big)\) with Radon--Nikodym cocycle \(\eta\), such that \(\widehat{X}\) is a \(G\)-invariant subset of \(\mathbb{G}\) and \(\eta=\mathfrak{g}\mid_{\widehat{X}}\).
\end{defn}

\subsection{Integrable harmonic majorant}

The following is a key step in the proof of Theorem~\ref{mthm:compactRNmodel}.

\begin{prop}\label{prop:Hcompact}
Let \(\left(G,\theta\right)\) be a measured group such that \(\mathbb{E}_{\theta}\left[\mathfrak{m}_{\theta}\right]<+\infty\) (and in particular, Harnack's inequality holds). Then the compact-open topology turns \(\mathbb{H}\left(G,\theta\right)\) into a compact metric space.
\end{prop}

\begin{lem}\label{lem:deltaK}
Let \(\mathcal{K}\) be the class of compact identity neighborhoods in \(G\), and for \(K\in\mathcal{K}\) define
\[\delta\left(K\right):=\sup\nolimits_{k\in K}\left\Vert \mathfrak m_{\theta}\cdot\left(\varphi-k\varphi\right)\right\Vert_{L^{1}\left(G\right)},\]
where \(\varphi:=d\theta/dm_{G}\). Then under the assumption \(\mathbb{E}_{\theta}\left[\mathfrak{m}_{\theta}\right]<+\infty\), one has
\[\delta\left(K\right)\xrightarrow[\mathcal K\ni K\searrow\{e_{G}\}]{}0.\]
\end{lem}

\begin{proof}[Proof of Lemma~\ref{lem:deltaK}]
By the assumption,
\[\mathbb{E}_{\theta}\left[\mathfrak{m}_{\theta}\right]=\int_{G}\mathfrak{m}_{\theta}\left(g\right)\varphi\left(g\right)dm_{G}\left(g\right)<+\infty,\]
and thus \(\mathfrak{m}_{\theta}\cdot\varphi\in L^{1}\left(G\right)\). Fix a compact identity neighborhood \(C\). Since \(\mathfrak{m}_{\theta}\) is locally bounded by Proposition~\ref{prop:mtheta}(2), let
\[M:=\sup\nolimits_{g\in C}\mathfrak m_{\theta}\left(g\right)<+\infty.\]
Given \(\epsilon>0\), choose a compact set \(L\) in \(G\) with \(\int_{G\setminus L}\mathfrak m_{\theta}\left(g\right)\varphi\left(g\right)dm_{G}\left(g\right)<\epsilon\). For every \(k\in C\) we have
\begin{align*}
\left\Vert \mathfrak m_{\theta}\cdot\left(\varphi-k\varphi\right)\right\Vert_{L^{1}\left(G\right)}
&\leq\int_{CL}\mathfrak m_{\theta}\left(g\right)\left|\varphi\left(g\right)-k\varphi\left(g\right)\right|dm_{G}\left(g\right)\\
&\qquad+\int_{G\setminus CL}\mathfrak m_{\theta}\left(g\right)\varphi\left(g\right)dm_{G}\left(g\right)
+\int_{G\setminus CL}\mathfrak m_{\theta}\left(g\right)\cdot k\varphi\left(g\right)dm_{G}\left(g\right).
\end{align*}
The first term is bounded by \(\sup\nolimits_{g\in CL}\mathfrak m_{\theta}\left(g\right)\cdot\left\Vert \varphi-k\varphi\right\Vert_{L^{1}\left(G\right)}\). Since \(CL\) is compact and \(\mathfrak{m}_{\theta}\) is locally bounded, and using the \(L^{1}\)-continuity of translations of \(\varphi\in L^{1}\left(G\right)\), there is an open identity neighborhood \(U\) so that \(\sup\nolimits_{g\in CL}\mathfrak m_{\theta}\left(g\right)\cdot\left\Vert \varphi-k\varphi\right\Vert_{L^{1}\left(G\right)}<\epsilon\) for all \(k\in U\). The second term is bounded by \(\epsilon\) since \(L\subseteq CL\). For the third term, by changing variables \(g=kh\) it is bounded by
\begin{align*}
\int_{G}\mathbf{1}_{G\setminus CL}\left(kh\right)\mathfrak m_{\theta}\left(kh\right)\varphi\left(h\right)dm_{G}\left(h\right)
&\leq M\cdot\int_{G}\mathbf{1}_{G\setminus CL}\left(kh\right)\mathfrak m_{\theta}\left(h\right)\varphi\left(h\right)dm_{G}\left(h\right)\\
&\leq M\cdot\int_{G\setminus L}\mathfrak m_{\theta}\left(h\right)\varphi\left(h\right)dm_{G}\left(h\right)<M\epsilon,
\end{align*}
where the first inequality uses that \(\mathfrak m_{\theta}\) is submultiplicative, and the second uses that \(h\in L\implies kh\in CL\). We conclude that \(\delta\left(K\right)<\left(M+2\right)\epsilon\) for every \(\mathcal{K}\ni K\subseteq U\). As \(\epsilon\) is arbitrary, this completes the proof.
\end{proof}

\begin{proof}[Proof of Proposition~\ref{prop:Hcompact}]
A routine Arzel\`{a}--Ascoli argument for \(\sigma\)-compact spaces shows that a closed set of \(C\left(G,\mathbb{R}_{\geq 0}\right)\) whose elements are uniformly bounded and equicontinuous on compact sets, is compact. We then show that \(\mathbb{H}\left(G,\theta\right)\) is such a set. Fix some compact set \(C\) in \(G\). The uniform boundedness of the elements of \(\mathbb{H}\left(G,\theta\right)\) on \(C\) is a consequence of the local boundedness of \(\mathfrak{m}_{\theta}\) as in Proposition~\ref{prop:mtheta}(2). For the equicontinuity on \(C\), let
\[M:=\sup\nolimits_{g\in C}\mathfrak{m}_{\theta}\left(g\right)<+\infty.\]
Given \(u\in\mathbb{H}\left(G,\theta\right)\) and \(h\in C\), for every \(K\in\mathcal{K}\) and \(k\in K\) we have the bound
\begin{align*}
\left|u\left(h\right)-u\left(hk\right)\right|
&=\Big|\int_{G}u\left(hg\right)\left(\varphi\left(g\right)-k\varphi\left(g\right)\right)dm_{G}\left(g\right)\Big|\\
&\leq\int_{G}u\left(hg\right)\left|\varphi\left(g\right)-k\varphi\left(g\right)\right|dm_{G}\left(g\right)\\
&\leq M\cdot\int_{G}\mathfrak{m}_{\theta}\left(g\right)\left|\varphi\left(g\right)-k\varphi\left(g\right)\right|dm_{G}\left(g\right)\leq M\cdot\delta\left(K\right),
\end{align*}
where the second inequality is by the aforementioned property of \(\mathfrak{m}_{\theta}\) and using that \(h\in C\), so that
\[u\left(hg\right)\leq u\left(h\right)\mathfrak{m}_{\theta}\left(g\right)\leq\mathfrak{m}_{\theta}\left(h\right)\mathfrak{m}_{\theta}\left(g\right)\leq M\cdot\mathfrak{m}_{\theta}\left(g\right)\text{ for all }g\in G.\]
Then the equicontinuity of the elements of \(\mathbb{H}\left(G,\theta\right)\) on \(C\) follows from Lemma~\ref{lem:deltaK}. In particular, since \(C\) was arbitrary, it follows that the elements of \(\mathbb{H}\left(G,\theta\right)\) are continuous, and \(\mathbb{H}\left(G,\theta\right)\subset C\left(G,\mathbb{R}_{\geq 0}\right)\).

We next claim that \(\mathbb{H}\left(G,\theta\right)\) is closed in \(C\left(G,\mathbb{R}_{\geq 0}\right)\) with the compact-open topology. Suppose \(\mathbb{H}\left(G,\theta\right)\ni u_{n}\to u\in C\left(G,\mathbb{R}_{\geq 0}\right)\) uniformly on compact sets. Since \(u_{n}\left(e_{G}\right)=1\) for all \(n\) then \(u\left(e_{G}\right)=1\); as for the positivity, for an arbitrary compact set \(C\), if we let \(M:=\sup_{g\in C}\mathfrak{m}_{\theta}\left(g^{-1}\right)\) then by Proposition~\ref{prop:mtheta}(3) one has \(u_{n}\mid_{C}\geq M^{-1}\) for all \(n\), hence also \(u\mid_{C}\geq M^{-1}\), and since \(C\) is arbitrary it follows that \(u>0\); finally, as for the \(\theta\)-harmonicity, fix \(g\in G\), and as we mentioned above,
\[u_{n}\left(gh\right)\leq\mathfrak{m}_{\theta}\left(h\right)u_{n}\left(g\right)\leq\mathfrak{m}_{\theta}\left(h\right)\mathfrak{m}_{\theta}\left(g\right)\,\,\text{ for all }n\text{ and }h\in G.\]
However, since \(\mathbb{E}_{\theta}\left[\mathfrak{m}_{\theta}\right]<+\infty\), the right-hand side is \(\theta\)-integrable as a function of \(h\), so the \(\theta\)-harmonicity of \(u\) follows by dominated convergence. This concludes that \(\mathbb{H}\left(G,\theta\right)\) is also closed in \(C\left(G,\mathbb{R}_{\geq 0}\right)\).
\end{proof}

\subsection{Compactly generated groups}

We now specialize in compactly generated group, where growth functions control Harnack's inequality and the integrability of the harmonic majorant.

\smallskip

Suppose \(G\) is an lcsc group generated (algebraically) by a compact symmetric identity neighborhood \(K\). The associated word length \(\left|\cdot\right|=\left|\cdot\right|_{K}\) is the subadditive function
\[\left|\cdot\right|:G\to\left[0,+\infty\right),\quad\left|g\right|:=\inf\{n\geq 0:g\in K^{n}\}.\]
For \(R>0\), denote by \(B_{R}\) the ball of radius \(R\) around \(e_{G}\) with respect to \(\left|\cdot\right|\), which is a compact set, and define the {\bf Harnack exponent} of \(\left(G,\theta\right)\) to be the (potentially infinite) constant
\[\gamma_{\theta}:=\limsup_{r\to+\infty}\tfrac{1}{r}\log\sup\left\{ \mathfrak{m}_{\theta}\left(g\right):g\in B_{r}\right\}\in\left[0,+\infty\right].\]
As word lengths are pairwise quasi-isometric,\footnote{If \(K^{\prime}\) is another compact symmetric identity neighborhoods generating \(G\) with word length \(\left|\cdot\right|^{\prime}\), there are constants \(\alpha\geq 1\) and \(\beta\geq0\) such that \(\alpha^{-1}\cdot\left|g\right|-\beta\leq\left|g\right|^{\prime}\leq\alpha\cdot\left|g\right|+\beta\) for all \(g\in G\).} the exponent \(\gamma_{\theta}\) is independent on the choice of \(K\).

\begin{prop}\label{prop:compgen}
Let \(\left(G,\theta\right)\) be a measured group as above.
\begin{enumerate}
    \item \(\left(G,\theta\right)\) satisfies Harnack's inequality if and only if \(\gamma_{\theta}<+\infty\).
    \item Suppose for every \(r>0\) there is a positive constant \(C\left(r\right)>0\) such that
    \[\varphi\left(g_{1}\right)\leq C\left(r\right)\varphi\left(g_{2}\right)\text{ whenever }\left|\left|g_{1}\right|-\left|g_{2}\right|\right|\leq r.\]
    Then \(\left(G,\theta\right)\) has Harnack's inequality.
    \item If \(\gamma_{\theta}<+\infty\) and \(\mathbb{E}_{\theta}\big[e^{\gamma\left|\cdot\right|}\big]<+\infty\) for some \(\gamma>\gamma_{\theta}\), then the harmonic majorant is integrable:
    \[\mathbb{E}_{\theta}\big[\mathfrak{m}_{\theta}\big]<+\infty.\]
\end{enumerate}
\end{prop}

\begin{proof}[Proof of Proposition~\ref{prop:compgen}]
For (1), for \(r>0\) let \(A\left(r\right):=\sup\left\{ \mathfrak{m}_{\theta}\left(g\right):g\in B_{r}\right\}\) and \(\Theta\left(r\right):=\log A\left(r\right)\). Since \(g\mapsto\mathfrak{m}_{\theta}\left(g\right)\) is submultiplicative, then so is \(r\mapsto A\left(r\right)\), hence \(r\mapsto\Theta\left(r\right)\) is subadditive, so \(\gamma_{\theta}\) is well-defined. Moreover,
\[\tfrac{1}{r}\Theta\left(r\right)\leq\Theta\left(1\right)=\log\sup\left\{ \mathfrak{m}_{\theta}\left(g\right):g\in K\right\}\text{ for all }r>0.\]
Therefore, Harnack's inequality implies \(\gamma_{\theta}\leq\Theta\left(1\right)<+\infty\). The fact that \(\gamma_{\theta}=+\infty\) otherwise is clear.

For (2), by Lemma~\ref{lem:suitableHarnack} it suffices to show that all \(G\)'s elements are \(\theta\)-suitable. Fix \(g\in G\). Then for every \(x\in G\), from \(\left|\left|g^{-1}x\right|-\left|x\right|\right|\leq\left|g\right|\) we get that
\[g\varphi\left(x\right)\leq C\left(\left|g\right|\right)\cdot\varphi\left(x\right).\]
This shows that \(g\theta\leq C\left(\left|g\right|\right)\cdot\theta\), and hence \(g\) is \(\theta\)-suitable with \(k=1\), \(R=1\), \(k_{1}=1\), and \(\alpha_{1}=C\left(\left|g\right|\right)\).

For (3), put \(\epsilon=\left(\gamma-\gamma_{\theta}\right)/2\), and pick \(r_{o}\geq 1\) such that \(\Theta\left(r\right)\leq \left(\gamma_{\theta}+\epsilon\right)r\) for all \(r\geq r_{o}\). Fix \(r\geq r_{o}\), and write \(r=qr_{o}+s\) with some integer \(q\geq 1\) and \(0\leq s<r_{o}\). Then
\[\Theta\left(r\right)\leq \Theta\left(qr_{o}\right)+\Theta\left(s\right)\leq \left(\gamma_{\theta}+\epsilon\right)qr_{o}+\Theta\left(s\right)\leq \left(\gamma_{\theta}+\epsilon\right)r+\Theta\left(r_{o}\right),\]
where the first inequality is by subadditivity, the second follows from the choice of \(r_{o}\), and the third from \(qr_{o}\leq r\) together with the monotonicity of \(\Theta\). If \(0\leq r<r_{o}\), then trivially
\[\Theta\left(r\right)\leq \Theta\left(r_{o}\right)\leq \left(\gamma_{\theta}+\epsilon\right)r+\Theta\left(r_{o}\right).\]
We thus found that for every \(r\geq 0\),
\[\Theta\left(r\right)\leq \left(\gamma_{\theta}+\epsilon\right)r+\Theta\left(r_{o}\right),\]
and hence for every \(g\in G\),
\[\mathfrak{m}_{\theta}\left(g\right)\leq A\left(\left|g\right|\right)=\exp\left(\Theta\left(\left|g\right|\right)\right)\leq \exp\left(\left(\gamma_{\theta}+\epsilon\right)\left|g\right|\right)\cdot A\left(r_{o}\right)\leq \exp\left(\gamma\left|g\right|\right)\cdot A\left(r_{o}\right).\]
It follows that
\[\mathbb{E}_{\theta}\left[\mathfrak{m}_{\theta}\right]\leq A\left(r_{o}\right)\cdot\mathbb{E}_{\theta}\big[e^{\gamma\left|\cdot\right|}\big]<+\infty.\qedhere\]
\end{proof}

\subsection{A universal compact model for the Radon--Nikodym factor}

The Radon--Nikodym factor was introduced by Kaimanovich--Vershik in the discrete case \cite[\S3.3]{kaimanovich1983random}, and by Nevo--Zimmer in the general case \cite[\S1.2]{nevo2000rigidity}. By definition, the Radon--Nikodym factor of a stationary \(\left(G,\theta\right)\)-space \(\left(X,\mu\right)\) is the sub-\(\sigma\)-algebra generated by the Radon--Nikodym derivatives \(\{\varrho_{g}\left(\cdot\right):g\in G\}\). By the cocycle property, this sub-\(\sigma\)-algebra is \(G\)-invariant, and therefore it forms a factor of \(\left(X,\mu\right)\) that can be realized measurably as a stationary \(\left(G,\theta\right)\)-space. Towards constructing the universal compact Radon--Nikodym model, we construct a topological realization of the Radon--Nikodym factor, this too in a universal compact way (a similar discussion in discrete countable groups can be found \cite[\S2.1]{sayag2025entropy}).

\smallskip

Suppose \(\left(G,\theta\right)\) satisfies \(\mathbb{E}_{\theta}\left[\mathfrak{m}_{\theta}\right]<+\infty\), and thus by Proposition~\ref{prop:Hcompact} the space \(\mathbb{H}\left(G,\theta\right)\) is a compact metric space with the compact-open topology. Turn \(\mathbb{H}\left(G,\theta\right)\) into a compact \(G\)-space by letting
\[g.u\left(h\right)=u\left(g^{-1}h\right)/u\left(g^{-1}\right).\]
This action is clearly continuous. There is a natural cocycle for this action, given by point evaluation:
\[\mathrm{ev}:G\times\mathbb{H}\to\mathbb{R}_{>0},\quad\mathrm{ev}_{g}\left(u\right)=u\left(g^{-1}\right).\]
It is straightforward to verify that this is a continuous cocycle.

\smallskip

In order to construct the desired factor map, we will need a version of the Radon--Nikodym cocycle whose fibers are elements of \(\mathbb{H}\left(G,\theta\right)\) on some \(\mu\)-conull set:

\begin{defn}\label{dfn:goodRN}
Let \(\left(X,\mu\right)\) be a stationary \(\left(G,\theta\right)\)-space. A measurable map \(\varrho:G\times X\to\mathbb{R}_{>0}\) is said to be a {\bf \(\theta\)-harmonic version} for the associated Radon--Nikodym cocycle, if the following hold:
\begin{enumerate}
    \item \(\varrho\) is a version of the associated Radon--Nikodym cocycle (recall Definition \ref{dfn:version}).
    \item There is a \(\mu\)-conull set \(X^{\ast}\) such that \(u_{x}:g\mapsto\varrho_{g^{-1}}\left(x\right)\) lies in \(\mathbb{H}\left(G,\theta\right)\) for every \(x\in X^{\ast}\).
\end{enumerate}
\end{defn}

\begin{lem}\label{lem:goodRN}
Every Radon--Nikodym cocycle of a stationary \(\left(G,\theta\right)\)-space admits a \(\theta\)-harmonic version.
\end{lem}

\begin{proof}[Proof of Lemma~\ref{lem:goodRN}]
Using the Mackey cocycle theorem, pick a strict version \(\varpi:G\times X\to\mathbb{R}_{>0}\) of the Radon--Nikodym cocycle of \(\left(X,\mu\right)\). Put \(\psi_{g}\left(x\right)=\varpi_{g^{-1}}\left(x\right)\), and by \(\theta\)-stationarity for every \(g\in G\) we have
\begin{equation}\label{eq:statpi}
\psi_{g}\left(x\right)=\int_{G}\psi_{gh}\left(x\right)d\theta\left(h\right)\text{ for }\mu\text{-a.e. }x\in X.
\end{equation}
Now letting \(\varphi=d\theta/dm_{G}\), define
\[u_{x}\left(g\right):=\int_{G}\psi_{h}\left(x\right)\varphi\left(g^{-1}h\right)dm_{G}\left(h\right).\]
For every fixed \(g\in G\), by \eqref{eq:statpi} we have
\[u_{x}\left(g\right)=\psi_{g}\left(x\right)\text{ for }\mu\text{-a.e. }x\in X.\]
Then by Fubini's theorem, there is a \(\mu\)-conull set \(X^{\ast}\subseteq X\) such that for every \(x\in X^{\ast}\),
\[u_{x}\left(g\right)=\psi_{g}\left(x\right)\text{ for }m_{G}\text{-a.e. }g\in G.\]
After intersecting \(X^{\ast}\) with the another \(\mu\)-conull set, we may assume also that \(u_{x}\left(e_{G}\right)=1\) for every \(x\in X^{\ast}\). Now for every \(x\in X^{\ast}\), since \(u_{x}\left(\cdot\right)=\psi_{\left(\cdot\right)}\left(x\right)\) \(m_{G}\)-a.e. we obtain
\[u_{x}\ast\theta\left(g\right)=\int_{G}u_{x}\left(h\right)\varphi\left(g^{-1}h\right)dm_{G}\left(h\right)=\int_{G}\psi_{h}\left(x\right)\varphi\left(g^{-1}h\right)dm_{G}\left(h\right)=u_{x}\left(g\right).\]
Thus \(u_{x}\in\mathbb{H}\left(G,\theta\right)\) for every \(x\in X^{\ast}\). Finally, the desired \(\theta\)-harmonic version is given by
\[\varrho:G\times X\to\mathbb{R}_{>0},\quad\varrho_{g}\left(x\right):=u_{x}\left(g^{-1}\right).\qedhere\]
\end{proof}

We can now establish the topological realization of the Radon--Nikodym factor.

\begin{prop}\label{prop:RNfact}
Let \(\left(G,\theta\right)\) be a measured group satisfying \(\mathbb{E}\left[\mathfrak{m}_{\theta}\right]<+\infty\). Then for every stationary \(\left(G,\theta\right)\)-space \(\left(X,\mu\right)\), there exists a topological realization of its Radon--Nikodym factor of the form
\[\rho:\left(X,\mu\right)\to\left(\mathbb{H}\left(G,\theta\right),\rho_{\ast}\mu\right).\]
\end{prop}

\begin{proof}[Proof of Proposition~\ref{prop:RNfact}]
Using Lemma~\ref{lem:goodRN}, fix a \(\theta\)-harmonic version \(\varrho\) of the Radon--Nikodym cocycle, with its associated \(\mu\)-conull set \(X^{\ast}\) as in Definition~\ref{dfn:goodRN}, and define the map
\[\rho:X^{\ast}\to\mathbb{H}\left(G,\theta\right),\quad\rho\left(x\right)\left(g\right):=u_{x}\left(g\right)=\varrho_{g^{-1}}\left(x\right).\]
Let us show that \(\rho\) is \(G\)-equivariant as a factor map, meaning that for every \(g\in G\),
\[\rho\left(g.x\right)=g.\rho\left(x\right)\text{ for }\mu\text{-a.e. }x\in X.\]
Fix \(g\in G\). Let \(D\) be some countable dense subset in \(G\). Since \(\varrho\) is a version of the Radon--Nikodym cocycle, for each \(h\in D\) we have the cocycle identity
\[\varrho_{h^{-1}g}\left(x\right)=\varrho_{h^{-1}}\left(g.x\right)\cdot \varrho_{g}\left(x\right)\text{ for }\mu\text{-a.e. }x\in X.\]
Intersecting over \(h\in D\), we obtain a \(\mu\)-conull set \(X_{g}\) with
\[\varrho_{h^{-1}}\left(g.x\right)=\tfrac{\varrho_{h^{-1}g}\left(x\right)}{\varrho_{g}\left(x\right)}\text{ for all }\left(h,x\right)\in D\times X_{g}.\]
Put \(X_{g}^{\prime}:=X^{\ast}\cap X_{g}\cap g^{-1}.X^{\ast}\). Then for every \(\left(h,x\right)\in D\times X_{g}^{\prime}\),
\[u_{g.x}\left(h\right)=\varrho_{h^{-1}}\left(g.x\right)=\tfrac{\varrho_{h^{-1}g}\left(x\right)}{\varrho_{g}\left(x\right)}=\tfrac{u_{x}\left(g^{-1}h\right)}{u_{x}\left(g^{-1}\right)}=g.u_{x}\left(h\right).\]
Both functions \(h\mapsto u_{g.x}\left(h\right)\) and \(h\mapsto g.u_{x}\left(h\right)\) are continuous and coincide on \(D\), hence \(u_{g.x}=g.u_{x}\), namely \(\rho\left(g.x\right)=g.\rho\left(x\right)\)
for every \(x\in X_{g}^{\prime}\). Finally, it is a standard fact that \(\mathbb{H}\left(G,\theta\right)\) with the compact-open topology has the Borel \(\sigma\)-algebra \(\mathcal{B}\left(\mathbb{H}\left(G,\theta\right)\right)\) generated by the evaluation maps \(\left\{\mathrm{ev}_{g}:g\in G\right\}\), hence
\[\rho^{-1}\left(\mathcal{B}\left(\mathbb{H}\left(G,\theta\right)\right)\right)
=\sigma\left(\mathrm{ev}_{g}\circ\rho:g\in G\right)
=\sigma\left(\varrho_{g}:g\in G\right),\]
which is the Radon--Nikodym sub-\(\sigma\)-algebra.
\end{proof}

\subsection{A universal compact Radon--Nikodym model theorem}

We shall now prove Theorem~\ref{mthm:compactRNmodel}.

\begin{proof}[Proof of Theorem~\ref{mthm:compactRNmodel}]
Let us start by defining the model. Recall the compact \(G\)-space \(\mathbb{H}\left(G,\theta\right)\) with its continuous point evaluation cocycle \(\mathrm{ev}\). Using the Mackey--Varadarajan compact model theorem~\ref{thm:Vara}, fix a universal compact \(G\)-space \(\mathbb{Y}\left(G\right)\). Define the compact \(G\)-space
\begin{equation}\label{eq:univmod}
\mathbb{G}\left(G,\theta\right):=\mathbb{H}\left(G,\theta\right)\times\mathbb{Y}\left(G\right),
\end{equation}
with the diagonal action, and with the evaluation cocycle on the first coordinate,
\begin{equation}\label{eq:univcoc}
\mathfrak{g}:G\times\mathbb{G}\left(G,\theta\right)\to\mathbb{R}_{>0},\quad\mathfrak{g}_{g}\left(u,y\right):=\mathrm{ev}_{g}\left(u\right)=u\left(g^{-1}\right).
\end{equation}

For the rest of the proof we will show that this is indeed a universal compact Radon--Nikodym model. Given a stationary \(\left(G,\theta\right)\)-space \(\left(X,\mu\right)\), let \(\rho:X\to\mathbb{H}\left(G,\theta\right)\) be a topological realization of the Radon--Nikodym factor as in Proposition~\ref{prop:RNfact}, and pick a compact model \(\kappa:X\to \mathbb{Y}\left(G\right)\) for \(X\) as a Borel \(G\)-space using Theorem \ref{thm:Vara}. Define the \(G\)-equivariant map
\[\iota:X\to\mathbb{G}\left(G,\theta\right),\quad \iota\left(x\right):=\left(\rho\left(x\right),\kappa\left(x\right)\right).\]
Then \(\iota\) is injective since \(\kappa\) is, and so with the measure \(\widehat{\mu}=\iota_{\ast}\mu\), we obtain that \(\big(\mathbb{G}\left(G,\theta\right),\widehat{\mu}\big)\) is a stationary \(\left(G,\theta\right)\)-space, and \(\iota:\left(X,\mu\right)\hookrightarrow\big(\mathbb{G}\left(G,\theta\right),\widehat{\mu}\big)\) is a \(G\)-equivariant measure preserving embedding. We complete the proof by showing that \(\mathfrak{g}\) is a version of the Radon--Nikodym cocycle of \(\left(\mathbb{G}\left(G,\theta\right),\widehat{\mu}\right)\), that is, for every \(g\in G\) it holds that
\begin{equation}\label{eq:RNid}
\tfrac{dg^{-1}\widehat{\mu}}{d\widehat{\mu}}\left(u,y\right)=u\left(g^{-1}\right)=\mathfrak{g}_{g}\left(u\right)\text{ for }\widehat{\mu}\text{-a.e. }\left(u,y\right)\in\mathbb{G}\left(G,\theta\right).
\end{equation}
Fix \(g\in G\). We start by proving that
\begin{equation}\label{eq:RNid1}
\varrho_{g}\left(\iota^{-1}\left(u,y\right)\right)=\mathfrak{g}_{g}\left(u\right)\,\,\text{ for }\widehat{\mu}\text{-a.e. }\left(u,y\right)\in\mathbb{G}\left(G,\theta\right);
\end{equation}
indeed, for every \(\left(u,y\right)\in\iota\left(X\right)\), which is a \(\widehat{\mu}\)-conull set, there is a unique \(x\in X\) such that
\[\left(u,y\right)=\iota\left(x\right)=\left(\rho\left(x\right),\kappa\left(x\right)\right)=\left(u_{x},\kappa\left(x\right)\right),\]
and therefore
\[\varrho_{g}\big(\iota^{-1}\left(u,y\right)\big)=\varrho_{g}\left(x\right)=u_{x}\left(g^{-1}\right)=u\left(g^{-1}\right)=\mathfrak{g}_{g}\left(u\right),\]
establishing \eqref{eq:RNid1}. It follows that for every \(f\in L^{1}\left(\widehat{\mu}\right)\),
\begin{align*}
\int_{\mathbb{G}\left(G,\theta\right)}f\left(g^{-1}.\left(u,y\right)\right)d\widehat{\mu}\left(u,y\right)
&=\int_{X}f\big(g^{-1}.\iota\left(x\right)\big)d\mu\left(x\right)
=\int_{X}f\big(\iota\left(g^{-1}.x\right)\big)d\mu\left(x\right)\\
&=\int_{X}f\big(\iota\left(x\right)\big)\varrho_{g}\left(x\right)d\mu\left(x\right)
=\int_{\iota\left(X\right)}f\left(u,y\right)\varrho_{g}\big(\iota^{-1}\left(u,y\right)\big)d\widehat{\mu}\left(u,y\right)\\
&=\int_{\iota\left(X\right)}f\left(u,y\right)\mathfrak{g}_{g}\left(u\right)d\widehat{\mu}\left(u,y\right)
=\int_{\mathbb{G}\left(G,\theta\right)}f\left(u,y\right)\mathfrak{g}_{g}\left(u\right)d\widehat{\mu}\left(u,y\right).
\end{align*}
where \(\iota^{-1}\) is well-defined on the \(\widehat{\mu}\)-conull set \(\iota\left(X\right)\). Then \eqref{eq:RNid} readily follows.
\end{proof}

\begin{proof}[Proof of Corollary~\ref{corr:compactRNmodel}]
When \(\left(G,\theta\right)\) is such that \(\varphi:=d\theta/dm_{G}\in L_{\mathrm{c}}^{p}\left(G\right)\) for some \(1<p\leq+\infty\), Harnack's inequality holds by Theorem~\ref{mthm:Harnack}, and thus \(\mathfrak{m}_{\theta}\) is finite valued. By Proposition~\ref{prop:mtheta}(2), once \(\mathfrak{m}_{\theta}\) is finite-valued then it is further locally bounded, and thus we found that \(\mathfrak{m}_{\theta}\) is bounded on the support of \(\theta\). Therefore \(\mathbb{E}_{\theta}\left[\mathfrak{m}_{\theta}\right]<+\infty\), so the proof is concluded by Theorem~\ref{mthm:compactRNmodel}.
\end{proof}

\bibliographystyle{acm}
\bibliography{References}

\end{document}